\documentclass
{amsart}

\usepackage[dvipdfm]{graphicx}

\usepackage{amssymb,amsmath,amsthm,amscd,amsfonts}
\theoremstyle{plain}
\newtheorem{theorem}{Theorem}[section]
\newtheorem{proposition}[theorem]{Proposition}

\newtheorem{lemma}[theorem]{Lemma}

\theoremstyle{definition}
\newtheorem{definition}[theorem]{Definition}

\theoremstyle{remark}
\newtheorem{example}[theorem]{Example}
\newtheorem{remark}[theorem]{Remark}
\begin{document}
\title[SLag and Lag.sss in cones over toric Sasaki manifolds]{Special Lagrangians and Lagrangian self-similar solutions in cones over toric Sasaki manifolds}
\author{Hikaru Yamamoto}
\address{Department of Mathematics, Tokyo Institute of Technology, 2-12-1,
O-okayama, Meguro, Tokyo 152-8551, Japan}
\email{yamamoto.h.ah@m.titech.ac.jp}
\date{\today}
\begin{abstract} 
We construct some examples of special Lagrangian submanifolds 
and Lagrangian self-similar solutions in almost Calabi--Yau cones over toric Sasaki manifolds. 
For example, for any integer $g\geq 1$, 
we can construct a real $6$-dimensional Calabi--Yau cone $M_{g}$ 
and a $3$-dimensional special Lagrangian submanifold $F^{1}_{g}:L_{g}^{1}\rightarrow M_{g}$ which is diffeomorphic to $\Sigma_{g} \times \mathbb{R}$ 
and a compact Lagrangian self-shrinker $F^{2}_{g}:L_{g}^{2}\rightarrow M_{g}$ which is  
diffeomorphic to $\Sigma_{g} \times S^1$, where $\Sigma_{g}$ is a closed surface of genus $g$. 
\end{abstract}

\maketitle
\section{Introduction}\label{Intro}
Special Lagrangian submanifolds are defined in almost Calabi-Yau manifolds. 
 Recently special Lagrangian submanifolds have acquired an important role in Mirror Symmetry. 
For example, they are key words in the Strominger--Yau--Zaslow Conjecture \cite{StromingerYauZaslow} which explains Mirror Symmetry of $3$-dimensional Calabi--Yau manifolds. 
Furthermore Thomas and Yau \cite{ThomasYau} introduced a stability condition for graded Lagrangians 
and conjectured that a stable Lagrangian converges to a special Lagrangian submanifold by the mean curvature flow. 

In this conjecture, the mean curvature flow is also one of important key words. 
Simply stated, mean curvature flows are gradient flows of volume functionals of manifolds. 
In a precise sense, it is a flow of a manifold in a Riemannian manifold moving along its mean curvature vector field. 
Let $(M,g)$ be a Riemannian manifold, $N$ a manifold and $F:N\times[0,T)\rightarrow M$ a smooth family of immersions, 
then $F$ is called a mean curvature flow if it satisfies
\begin{align*}
\frac{\partial F}{\partial t}(p,t) = H_{t}(p) \hspace{5mm}\mathrm{~for~all~}(p, t) \in N\times [0, T)
\end{align*}
where $H_{t}$ is the mean curvature vector field of the immersion $F_{t} := F(\cdot, t) : N\rightarrow M$. 
If the ambient is $\mathbb{R}^{m}$, there is an important class of solutions called {\it self-similar solution}. 
An immersion of a manifold $F:N\rightarrow\mathbb{R}^m$ is called a self-similar solution if it satisfies 
\begin{align*}
H=\lambda F^{\bot}
\end{align*}
where $\lambda \in \mathbb{R}$ is a constant and $F^{\bot}$ is the normal part of the position vector $F$. 
Huisken \cite{Huisken} has studied mean curvature flows in $\mathbb{R}^m$ 
and proved that if the mean curvature flow in $\mathbb{R}^m$ has the type I singularity, 
then there exists a smoothly convergent subsequence of the rescaling such that its limit becomes a self-similar solution. 
In this sense, a self-similar solution can be thought of as an asymptotical  model of a mean curvature flow which develops a type I singularity at the time when it blowups. 

In this paper, we construct Lagrangian self-similar solutions in {\it cone manifolds}. 
To define self-similar solutions in cone manifolds, we use the generalization of position vectors in $\mathbb{R}^m$ to cone manifolds defined by Futaki, Hattori and the author in \cite{FutakiHattoriYamamoto}. 

Here we introduce some notations over cone manifolds. 
First, for a Riemannian manifold $(S,g)$, we say that $(C(S),\overline{g})$ is a cone over $(S,g)$, 
if $C(S) \cong S\times\mathbb{R}^{+}$ and $\overline{g}=r^2g+dr^2$ where $r$ is the standard coordinate of $\mathbb{R}^+$. 
We denote two projections by $\pi:C(S)\rightarrow S$ and $r:C(S)\rightarrow \mathbb{R}^{+}$. 
On the cone $C(S)$, there is a natural $\mathbb{R}^{+}$-action defined below. 
This action can be considered as an expansion or shrinking on the cone. 
\begin{definition}\label{raction}
We define the $\mathbb{R}^{+}$-action on $C(S)$ by 
\begin{align*}
\rho \cdot p_{0} =(s_{0},\rho r_{0})\quad \in C(S)\cong S\times\mathbb{R}^{+}
\end{align*}
for all $\rho \in \mathbb{R}^{+}$ and $p_{0}=(s_{0},r_{0}) \in C(S)$. 
\end{definition}

\begin{definition}
For a point $p_{0}=(s_{0},r_{0}) \in S\times\mathbb{R}^{+} \cong C(S)$, we define the position vector $\overrightarrow{p_{0}}$ by 
$$\overrightarrow{{p}_{0}}=r_{0}\frac{\partial}{\partial r}\bigg|_{r=r_{0}}\quad \in T_{p_{0}}C(S). $$
Furthermore, for a map $F:N\rightarrow C(S)$ from a manifold $N$, we define the position vector $\overrightarrow{F}$ of $F$ by $\overrightarrow{F}(x):=\overrightarrow{F(x)}$  at $x\in N$. 
Note that $\overrightarrow{F}$ is a section of $F^{*}(TC(S))$ over $N$. 
\end{definition}
Clearly $\overrightarrow{p_{0}}$ coincides with the derivative of the curve $c(\rho):=\rho\cdot p_{0}$ in $C(S)$ at $\rho=1$, that is, 
$$\overrightarrow{p_{0}}=\frac{d}{d\rho}\biggl |_{\rho=1}(\rho\cdot p_{0}). $$
Using this generalization of the position vector, we can define self-similar solutions in cone manifolds. 
\begin{definition}
Let $N$ be a manifold. 
An immersion $F:N \rightarrow C(S)$ is called a self-similar solution if 
\begin{align*}
 H = \lambda {\overrightarrow{F}}^{\bot}
\end{align*}
where $\lambda \in \mathbb{R}$ is a constant. 
It is called a self-shrinker if $\lambda<0$ and self-expander if $\lambda>0$. 
\end{definition}
Here $\bot$ is the orthogonal projection map from $F^{*}(TC(S))$ to $T^{\bot}N$ which is an orthogonal complement of $F_{*}(TN)$. 
Furthermore if a self-similar solution in a K\"ahler manifold is a Lagrangian submanifold, then we call it a Lagrangian self-similar solution.  

The typical results in $\mathbb{R}^n$ studied by Huisken  \cite{Huisken} are extended to the mean curvature flow in a cone manifold by Futaki, Hattori and the author in \cite{FutakiHattoriYamamoto}. 
For example, it is proved in \cite{FutakiHattoriYamamoto} that if a mean curvature flow in a cone manifold has the type $\mathrm{I}_{c}$ singularity, 
then there exists a smoothly convergent subsequence of the rescaling such that its limit becomes a self-similar solution. 
Type $\mathrm{I}_{c}$ singularity is a certain kind of singularity similar to type I singularity, and for more details refer to \cite{FutakiHattoriYamamoto}. 

In this paper, we present a method of constructing special Lagrangian submanifolds and Lagrangian self-similar solutions in toric Calabi--Yau cones. 
First we construct Lagrangian submanifolds in toric K\"ahler cone in Theorem \ref{Lag}. 
Next, if the canonical line bundle of the toric K\"ahler cone is trivial, that is, it is a toric almost Calabi--Yau cone, then 
we construct special Lagrangian submanifolds in Theorem \ref{54321} and Theorem \ref{12345}, and 
Lagrangian self-similar solutions in Theorem \ref{11223}. 
These constructions are considered to be a kind of extension of special Lagrangian submanifolds in $\mathbb{C}^m$ by Harvey and Lawson \cite{HarveyLawson} and 
Lagrangian self-similar solutions in $\mathbb{C}^m$ by Joyce, Lee and Tsui in \cite{JoyceLeeTsui}, see Remark \ref{HL} and Remark \ref{JLT123}. 
As an application of these theorems, we concretely construct some examples. 
\begin{example}[cf. Example \ref{ex1}]\label{ex-intro}
For any integer $g\geq 1$, 
we construct a real $6$-dimensional Calabi--Yau cone $M_{g}$ 
and a $3$-dimensional special Lagrangian submanifold $F^{1}_{g}:L_{g}^{1}\rightarrow M_{g}$ which is diffeomorphic to $\Sigma_{g} \times \mathbb{R}$ 
and a compact Lagrangian self-similar solution (self-shrinker) $F^{2}_{g}:L_{g}^{2}\rightarrow M_{g}$ which is  
diffeomorphic to $\Sigma_{g} \times S^1$ concretely, where $\Sigma_{g}$ is a closed surface of genus $g$. 
\end{example}

This paper is organized as follows. 
In Section \ref{TS}, we introduce some basic definitions and propositions in toric Sasaki manifolds. 
In Section \ref{CL}, we construct Lagrangian submanifolds in cones over toric Sasaki manifolds. 
In Section \ref{aCY}, we explain some details about almost Calabi--Yau manifolds, Lagrangian angles, special Lagrangian submanifolds and generalized mean curvature vectors. 
In Section \ref{LA}, we compute the Lagrangian angles of Lagrangians constructed in Section \ref{CL} when the ambient is a toric almost Calabi--Yau cone. 
Section \ref{constSLag} is devoted to the proofs of Theorem \ref{54321} and \ref{12345}. 
Section \ref{constLss} is devoted to the proofs of Theorem \ref{11223}. 
In Section \ref{TKC}, for an application of our theorems, we construct some concrete examples in toric Calabi--Yau $3$-folds. \\
 \subsection*{Acknowledgements}
 I would like to thank to A. Futaki for introducing me to the subject of special Lagrangian geometry, 
 for many useful suggestions and discussions concerning Sasakian geometry and for his constant encouragement.   
\section{Toric Sasaki manifold}\label{TS}
In this section we introduce some definitions and propositions in toric Sasaki manifolds. 
Proofs of the results in this section are summarized in the papers of Boyer and Galicki \cite{BoyerGalicki} and Martelli, Sparks and Yau \cite{MartelliSparksYau}. 
First of all, we define Sasaki manifolds. 
\begin{definition}\label{Sasaki}
Let $(S, g)$ be a Riemannian manifold and $\nabla$ the Levi-Civita connection of the Riemannian metric $g$. 
Then $(S,g)$ is said to be a Sasaki manifold if and only if 
it satisfies one of the following two equivalent conditions. 
\begin{itemize}
\item [(\ref{Sasaki}.a)] There exists a Killing vector field $\xi$ of unit length on $S$
so that the tensor field $\Phi$ of type $(1,1)$, defined by
$\Phi(X) ~=~ \nabla_X \xi$, satisfies 
\begin{align*}
(\nabla_X\Phi)(Y) = g(\xi,Y)X - g(X,Y)\xi.
\end{align*}
\item[(\ref{Sasaki}.b)] 
There exists a complex structure $J$ on $C(S)$ compatible with $\overline{g}$ so that 
$(C(S),{\bar g},J)$ becomes a K\"ahler manifold. 
\end{itemize}
\end{definition}
We call the quadruple $(\xi,\eta,\Phi,g)$ on $S$ the Sasaki structure. 
$S$ is often identified with the submanifold $\{r = 1\} =  S  \times \{1\}\subset C(S)$.
By the definition, the dimension of $S$ is odd and denoted by $2m - 1$. 
Hence the complex dimension of $C(S)$ is  $m$. 
Note that $C(S)$ does not contain the apex. 

The equivalence of (\ref{Sasaki}.a) and (\ref{Sasaki}.b) can be seen as follows. 
If $(S,g)$ satisfies the condition (\ref{Sasaki}.a), we can define a complex structure $J$ on $C(C)$ as 
$$JY = \Phi(Y) - \eta(Y) r\frac{\partial}{\partial r}\quad \mathrm{and} \quad J r\frac{\partial}{\partial r} = \xi. $$ 
for all $Y \in \Gamma(TS)$ and $r(\partial/\partial r) \in \Gamma(T\mathbb{R}^{+})$, where $\eta$ is a 1-form on $S$ defined by $\eta(Y)=g(\xi,Y)$. 
Conversely, if $(S,g)$ satisfies condition (\ref{Sasaki}.b), we have a Killing vector field $\xi$ defined as 
$\xi= J \frac{\partial}{\partial r}$. 

We can extend $\xi$ and $\eta$ also on the cone $C(S)$ by putting
$$ \xi = Jr\frac{\partial}{\partial r}, \qquad  \eta(Y) = \frac 1{r^2}\overline{g}(\xi, Y) $$
where $Y$ is any smooth vector field on $C(S)$. 
Of course $\eta$ on $C(S)$ is the pull-back of $\eta$ on $S$ by the projection $\pi:C(S) \rightarrow S$.
Furthermore the $1$-form $\eta$ is expressed on $C(S)$ as 
\begin{align}\label{eta}
\eta = 2d^c\log r
\end{align}
where $d^c=\frac i2 (\bar \partial-\partial)$. 
From (\ref{eta}), the K\"ahler form $\omega$ of the cone $(C(S),\overline{g})$ is expressed as
\begin{align}\label{kahler-form}
\omega=\frac 12 d(r^2\eta) = \frac12 d d^c r^2=\frac{i}{2}\partial\overline{\partial}r^2.
\end{align}
Remember that we have defined $\mathbb{R}^{+}$-action on $C(S)$ in Definition \ref{raction}. 
By (\ref{kahler-form}), it is clear that $\rho^{*}\omega=\rho^2\omega$, 
where we denote the transition map with respect to $\rho\in\mathbb{R}^{+}$ by the same symbol $\rho:C(S)\rightarrow C(S)$; $\rho(p)=\rho\cdot p$. 
Next, we introduce the notion of toric Sasaki manifolds. 
\begin{definition}\label{ToricSasaki}
A Sasaki manifold with Sasaki structure $(S, \xi,\eta,\Phi,g)$ of dimension $2m-1$ is a {\it toric} Sasaki manifold 
if and only if it  satisfies one of the following two equivalent conditions. 
\begin{itemize}
\item[(\ref{ToricSasaki}.a)]There is an effective action of $m$-dimensional torus $T^m$ on $S$ preserving the Sasaki structure.  
\item[(\ref{ToricSasaki}.b)]There is an effective holomorphic action of $m$-dimensional torus $T^m$ on $C(S)$ preserving $\overline{g}$. 
Furthermore two projections $\pi:C(S)\rightarrow S$ and $r:C(S)\rightarrow \mathbb{R}^{+}$ satisfy $\pi(\tau\cdot p)=\tau\cdot \pi(p)$ and $r(\tau\cdot p)=r(p)$ for all $\tau \in T^m$ and $p \in C(S)$. 
\end{itemize}
\end{definition}
It is clear that $\mathbb{R}^{+}$-action and $T^{m}$-action is commutative. 
The most typical example of the toric Sasaki manifold is the sphere $S^{2m-1}$, because $C(S)=\mathbb{C}^m\setminus\{0\}$ is toric K\"ahler. 

The equivalence of (\ref{ToricSasaki}.a) and (\ref{ToricSasaki}.b) can be seen as follows. 
If a Sasaki manifold $(S,g)$ satisfies the condition (\ref{ToricSasaki}.a), let $\tau\in T^m$ act on $C(S)$ as 
\begin{align*}
\tau \cdot p_{0} = (\tau\cdot  s_{0} , r_{0}) 
\end{align*}
for all $p_{0}=(s_{0},r_{0})\in C(S)$. Then this action on $C(S)$ satisfies the condition (\ref{ToricSasaki}.b). 
Conversely, if a Sasaki manifold $(S,g)$ satisfies the condition (\ref{ToricSasaki}.b), 
then the restriction of $T^m$-action to S satisfies the condition (\ref{ToricSasaki}.a). 

Let $\mathfrak{g}\cong\mathbb{R}^m$ be the Lie algebra of $T^m$ and $\mathfrak{g}^{*}$ be the dual vector space. 
We identify the vector field on $C(S)$ generated by $v\in\mathfrak{g}$ and $v$ itself. 
That is, for $p\in C(S)$ we write
$$v(p)=\frac{d}{dt}\bigg|_{t=0}\exp(tv)\cdot p. $$

A toric Sasaki manifold and its cone have a moment map 
$\mu:C(S)\rightarrow \mathfrak{g}^{*}$ with respect to the K\"ahler form $\omega=\frac{1}{2} d(r^2\eta)$. 
It is given by 
\begin{align}\label{momentmap}
\langle \mu(p),v \rangle =\frac{1}{2}r^2(p) \eta(v(p)),
\end{align}
for all $p \in C(S)$ and $v \in \mathfrak{g}$ and it satisfy
$$d\langle \mu ,v\rangle=-\omega(v,\cdot).$$ 

On the other hand, since $C(S)$ is a toric variety, 
there exists a fan $\Sigma$ of $C(S)$ and the complex structure on $C(S)$ is determined by $\Sigma$. 
Moreover there exists an $m$-dimensional complex torus $T^m_{\mathbb{C}}(\cong (\mathbb{C}^{\times})^{m})$ contains $T^{m}$ as a compact subgroup, 
and $T^{m}_{\mathbb{C}}$ acts on $C(S)$ as a bi-holomorphic automorphism and has an open dense $T^{m}_{\mathbb{C}}$-orbit. 
Hence, over $C(S)$, there exists an intrinsic anti-holomorphic involution $\sigma:C(S)\rightarrow C(S)$ determined by $\Sigma$, 
that is, $\sigma^2=id$ and $\sigma_{*} J=-J \sigma_{*}$. 
This involution satisfies 
\begin{align}\label{toric-sigma}
\sigma(w\cdot p)=\overline{w}\cdot \sigma(p), 
\end{align}
where $w\in T^m_{\mathbb{C}}$ and $p\in C(S)$. 
We denote the set of fixed points of $\sigma$ by 
$$C(S)^{\sigma}=\{\, p\in C(S)\mid \sigma(p)=p\,\}. $$ 
Then it is a real $m$-dimensional submanifold of $C(S)$, and we call it a real form of $C(S)$. 
Now we consider some properties of $\sigma$ and $C(S)^{\sigma}$. 

\begin{proposition}\label{anti-symp}
The involution $\sigma:C(S)\rightarrow C(S)$ is anti-symplectic. 
Thus it is also isometry. 
\end{proposition}
\begin{proof}
Let $U_{0}$ be an open dense $T^{m}_{\mathbb{C}}$-orbit. 
For $(w^{1},\dots,w^{m})\in U_{0}\cong T^{m}_{\mathbb{C}}\cong(\mathbb{C}^{\times})^m$, 
we take a logarithmic holomorphic coordinates $(z^{1},\dots,z^{m})$ defined by $e^{z^k}=w^k$. 
Since $\omega$ is $T^{m}$-invariant and the action of $T^{m}$ is Hamiltonian, 
there exists a function $F(x)\in C^{\infty}(\mathbb{R}^m)$ with the property
\begin{align*}
\omega=\frac{i}{2}\sum_{k,\ell=1}^{m} \frac{\partial^{2} F}{\partial x^{k} \partial x^{\ell}}dz^{k}\wedge d\overline{z}^{\ell}\quad\mathrm{on}\ U_{0},
\end{align*}
where $z^{k}=x^{k}+iy^{k}$. 
(See Guillemin \cite{Guillemin}.)
On $U_{0}$, the involution $\sigma$ coincides with the standard complex conjugate $\sigma(z)=\overline{z}$, 
where $\overline{z}=(\overline{z}^{1},\dots,\overline{z}^{m})$. 
Note that $F$ is independent of the coordinates $(y^{k})_{k=1}^{m}$. 
Thus we have $\sigma^{*}\omega=-\omega$ on $U_{0}$. 
Since $U_{0}$ is open and dense in $C(S)$, thus we have $\sigma^{*}\omega=-\omega$ on $C(S)$. 
Second statement follows immediately by combining the property that $\sigma$ is anti-holomorphic. 
\end{proof}
Here we have some remarks. 
\begin{remark}\label{abcd}
Take a point $p$ in real form $C(S)^{\sigma}$ and two vectors $X,Y$ in $T_{p}C(S)^{\sigma}$. 
Since $\sigma_{*}X=X$ and $\sigma_{*}Y=Y$,  
we have 
$$\omega(X,Y)=\omega(\sigma_{*}X,\sigma_{*}Y)=-\omega(X,Y)$$ 
by Proposition \ref{anti-symp}, hence $\omega=0$ on $C(S)^{\sigma}$. 
This means that the real form $C(S)^{\sigma}$ is a Lagrangian submanifold in $C(S)$. 
Moreover if we apply the condition (\ref{toric-sigma}) for $p$ and $\tau\in T^m$, 
we have $\sigma(\tau \cdot p)=\tau^{-1}\cdot p$, 
hence for all $v \in\mathfrak{g}$ we have $\sigma_{*}v(p)=-v(p)$. 
This means that $v(p)$ is orthogonal to $T_{p}C(S)^{\sigma}$ with respect to $\overline{g}$. 
\end{remark}

In general we do not know for $p$ in $C(S)^{\sigma}$ whether its position vector $\overrightarrow{p}$ is tangent to $C(S)^{\sigma}$. 
However if we assume the Reeb field $\xi$ is generated by an element in $\mathfrak{g}$, then it is ensured. 
For such a toric Sasaki manifold, we identify the Reeb vector field $\xi$ and an element in $\mathfrak{g}$ that generates $\xi$. 
\begin{proposition}\label{real-form}
Let $(S, \xi,\eta,\Phi,g)$ be a toric Sasaki manifold. 
If the Reeb field $\xi$ is generated by an element in $\mathfrak{g}$, 
then for all $p$ in $C(S)^{\sigma}$ its position vector $\overrightarrow{p}$ is tangent to $C(S)^{\sigma}$. 
\end{proposition}
\begin{proof}
Remember Remark \ref{abcd}. 
Since $C(S)^{\sigma}$ is a Lagrangian submanifold, we have orthogonal decomposition 
$$T_{p}C(S)=T_{p}C(S)^{\sigma}\oplus J(T_{p}C(S)^{\sigma}), $$
with respect to $\overline{g}$. 
Now $\xi$ is in $\mathfrak{g}$, hence $\xi(p)$ is orthogonal to $T_{p}C(S)^{\sigma}$, that is, $\xi(p)$ is in $J(T_{p}C(S)^{\sigma})$. 
On the other hand, $\xi(p) = J(r\frac{\partial}{\partial r})|_{\mathrm{at}\, p}=J(\overrightarrow{p})$. 
Thus we have $\overrightarrow{p} \in T_{p}C(S)^{\sigma}$. 
\end{proof}
In our paper we always assume that the Reeb field $\xi$ of toric Sasaki manifold is generated by an element in $\mathfrak{g}$. 
By Proposition \ref{real-form}, it follows that $C(S)^{\sigma}$ is also a cone manifold. 
If we write $S^{\sigma}=\{\, p\in S \mid \sigma(p)=p \,\}$, then $C(S)^{\sigma}=C(S^{\sigma})$. 

In the last of this section, we remark some facts that is well known in the toric contact geometry and the algebraic toric geometry. 
Let $C(S)$ be the cone of a toric Sasaki manifold $S$ with dimension $2m-1$ and with the Reeb field $\xi$. 
Let $\mathbb{Z}_{\mathfrak g}\cong \mathbb{Z}^m$ be the integral lattice of $\mathfrak g$, that is the kernel of the exponential map $\exp : \mathfrak g \to T^m$. 
Let $\Sigma$ be a fan of $C(S)$ and $\Lambda=\{\lambda_{1},\dots,\lambda_{d}\}\subset\mathbb{Z}_{\mathfrak{g}}$ be the primitive generators of 
the $1$-dimensional cones of $\Sigma$.  
Let $\Delta=\mu(C(S))$ be a moment image of $C(S)$ and let $\Delta_{0}^{*}$ be a (open) dual cone of $\Delta$ defined by
$$\Delta_{0}^{*}:=\{\, x \in \mathfrak{g} \mid \langle y, x \rangle >0\ \mathrm{for\ all}\ y \in \Delta \,\}. $$
\begin{remark}\label{remarkSasaki}
In fact, $\Delta$ is a {\it good} rational polyhedral cone defined below and the Reeb field $\xi$ is an element of $\Delta_{0}^{*}$. 
\end{remark}
The second statement in Remark \ref{remarkSasaki} is clear since for all $p$ in $C(S)$ we have
$$\langle \mu(p), \xi \rangle =\frac{1}{2}r^2(p) \eta(\xi(p))=\frac{1}{2}r^2(p)>0. $$

\begin{definition}[Good cone, cf. \cite{Lerman}] \label{good}
First we say that a subset $\Delta \subset \mathfrak{g}^{*}$ is a {\it rational polyhedral cone} 
if there exists a finite set of primitive vectors $\Lambda=\{\lambda_{1},\dots,\lambda_{d}\}\subset \mathbb{Z}_{\mathfrak{g}}$ such that
$$ \Delta = \{\,y \in \mathfrak{g}^{*} \mid  \langle y, \lambda \rangle \geq 0\, \mathrm{for\ }\,  \lambda\in\Lambda\, \}-\{0\}.$$
We assume that the set $\Lambda$ is minimal, that is, we can not express $\Delta$ by any subset $\Lambda'\subset\Lambda$, $\Lambda'\neq\Lambda$. 
Furthermore we say that $\Delta$ is {\it strongly convex} if $\Delta\cup\{0\}$ does not contain any straight lines of the form $\ell=\{\, p+vt \mid t\in\mathbb{R}\, \}$ for some $p$ and $v$ in $\mathfrak{g}^{*}$. 
Under these assumptions a strongly convex rational polyhedral cone $\Delta$ with non-empty interior is {\it good} if the following condition holds. 
If a subset $\Lambda'\subset\Lambda$ satisfies 
$$ \{\, y \in \Delta \mid \langle y, \lambda \rangle = 0\ \mathrm{for}\ \lambda\in\Lambda' \,\}\neq\emptyset, $$
then $\Lambda'$ is linearly independent over $\mathbb{Z}$ and 
\begin{equation}\label{goodcondition}
\biggl\{\, \sum_{\lambda\in\Lambda'}a_{\lambda}\lambda\,  \bigg|\, \ a_{\lambda} \in \mathbb{R}\, \biggr\} \cap \mathbb{Z}_{\mathfrak{g}} = 
\biggl\{\, \sum_{\lambda\in\Lambda'}m_{\lambda} \lambda \, \bigg|\, m_{\lambda} \in \mathbb{Z}\, \biggr\}.
\end{equation}
\end{definition}

By the standard algebraic toric geometry theory, 
we know that the canonical line bundle $K_{C(S)}$ of $C(S)$ is trivial or not. 
That is the following remark. 
\begin{remark}\label{trivial}
The canonical line bundle $K_{C(S)}$ of $C(S)$ is trivial if and only if 
there exists an element $\gamma\in(\mathbb{Z}_{\mathfrak{g}})^{*}\cong\mathbb{Z}^{m}$ such that 
$$\langle \gamma, \lambda \rangle = 1$$
for all $\lambda\in\Lambda$. 
In fact, by using this element $\gamma=(\gamma_{1},\dots,\gamma_{m})$, we can construct canonical non-vanishing holomorphic $(m,0)$-form on $C(S)$ 
by purely algebraic toric geometry way, and we denote it by $\Omega_{\gamma}$. 
On the open dense $T_{\mathbb{C}}^{m}$-orbit $U_{0}\cong(\mathbb{C}^{\times})^{m}$, 
we can express $\Omega_{\gamma}$ by the logarithmic holomorphic coordinates $(z^{k})_{k=1}^{m}$ by 
$$\Omega_{\gamma}=\exp(\gamma_{1}z^{1}+\dots+\gamma_{m}z^{m})dz^{1}\wedge\dots\wedge dz^{m}. $$
\end{remark}
\section{Construction of Lagrangian submanifolds}\label{CL}
Let $(S,g)$ be a toric Sasaki manifold with $\dim_{\mathbb{R}}S=2m-1$ and $(C(S),\overline{g})$ be the toric K\"ahler cone. 
In this section we construct the explicit examples of Lagrangian submanifolds in $C(S)$. 
Let $\mu:C(S)\rightarrow\mathfrak{g}^{*}$ be a moment map and $\Delta=\mu(C(S))$ be the moment image of $C(S)$. 
As explained in Section \ref{TS}, there exists a finite set of primitive vectors $\Lambda=\{\lambda_{1},\dots,\lambda_{d}\}\subset \mathbb{Z}_{\mathfrak g}$ such that
$$ \Delta = \{\, y \in \mathfrak{g}^{*} \mid \langle y, \lambda \rangle \geq 0\, \mathrm{for}\, \lambda\in\Lambda\,\}-\{0\}. $$

To construct Lagrangian submanifolds, first of all, take $\zeta\in\mathfrak{g}$ and $c\in\mathbb{R}$, and
we denote the hyperplane $\{\, y \in\mathfrak{g}^{*}  \mid \langle y, \zeta\rangle=c \,\}$ by $H_{\zeta,c}$. 
We assume that 
\begin{align}
& \mathrm{Int}\Delta\cap H_{\zeta,c}\neq\emptyset \quad \mathrm{and} \label{ass1}\\
& \zeta\notin\mathfrak{z}_{y}\quad \mathrm{for\ any\ }y\in\Delta\cap H_{\zeta,c}, \label{ass2}
\end{align}
where we define $\mathfrak{z}_{y}$ for $y\in \Delta$ by 
$$\mathfrak{z}_{y}=\mathrm{Span}_{\mathbb{R}}\{\,  \lambda_{i}  \mid  \langle y, \lambda_{i} \rangle=0  \}. $$
For example, if $y \in \mathrm{Int}\Delta$ then $\mathfrak{z}_{y}=\{0\}$. 
We denote the intersection of $\Delta$ and $H_{\zeta,c}$ by 
$$\Delta_{\zeta,c}=\Delta\cap H_{\zeta,c}.$$ 
First assumption (\ref{ass1}) 
means that $\Delta_{\zeta,c}$ is codimension one in $\Delta$. 
Second assumption (\ref{ass2}) means that if $p\in C(S)$ is in $ \mu^{-1}(\Delta_{\zeta,c})$ then $\zeta(p)\neq0$, 
where we identify $\zeta\in\mathfrak{g}$ and the vector field on $C(S)$ generated by $\zeta\in\mathfrak{g}$. 

Let $\sigma:C(S)\rightarrow C(S)$ be the involution explained in Section \ref{TS} and $C(S)^{\sigma}$ be the real form. 
Let $\mu^{\sigma}:C(S)^{\sigma}\rightarrow \Delta$ be the restriction of $\mu$ on the real form. 
In fact, $\mu^{\sigma}$ is a $2^{m}$-fold ramified covering of $\Delta$. 
We define a subset of $C(S)^{\sigma}$ as the pull-back of $\Delta_{\zeta,c}$ by $\mu^{\sigma}$ by
\begin{align*}
C(S)^{\sigma}_{\zeta,c}& =(\mu^{\sigma})^{-1}(\Delta_{\zeta,c})\\
& =\{\, p\in C(S)^{\sigma} \mid \langle\mu(p), \zeta \rangle=c \,\}. 
\end{align*}
By the assumptions (\ref{ass1}) and (\ref{ass2}), in fact $C(S)^{\sigma}_{\zeta,c}$ is a real $(m-1)$-dimensional submanifold in the real form $C(S)^{\sigma}$. 
Since $\mu^{\sigma}$ is a $2^{m}$-fold covering of $\Delta$, $C(S)^{\sigma}_{\zeta,c}$ is a $2^m$-fold covering of $\Delta_{\zeta,c}$. 

\begin{remark}
If $\zeta$ and $c$ do not satisfy  the assumptions (\ref{ass1}) and (\ref{ass2}), then $C(S)^{\sigma}_{\zeta,c}$ may become a singular submanifold. 
\end{remark}

To construct a Lagrangian submanifold, we move $C(S)^{\sigma}_{\zeta,c}$ by a one parameter action of $\mathbb{R}^{+}$ and torus $T^{m}$. 
Take an open interval $I \subset \mathbb{R}$. 
Let $f:I\rightarrow \mathbb{R}$ and $\rho:I \rightarrow \mathbb{R}^{+}$ be two functions on $I$, and $\tau_{0}$ be an element of torus $T^m$. 
We assume that $\dot{f}$ is non-vanishing on $I$. 
We denote the $1$-parameter orbit $\{\exp(f(t)\zeta)\cdot \tau_{0}\}_{t\in I}$ in torus by $\{\tau(t)\}_{t\in I}$. 
We define a real $m$-dimensional manifold by 
$$L_{\zeta,c}=C(S)^{\sigma}_{\zeta,c}\times I. $$ 
\begin{definition}
We define a map $F:L_{\zeta,c}\rightarrow C(S)$ by 
$$F(p,t):=\rho(t)\cdot \tau(t)\cdot p$$
for $(p,t)\in C(S)^{\sigma}_{\zeta,c}\times I =L_{\zeta,c}.$ 
\end{definition}

\begin{remark}\label{S1}
If  $\rho(t)\cdot \tau(t)$ is defined on $I=\mathbb{R}$ and periodic, then we can reduce $I$ to $S^{1}$ and take $L_{\zeta,c}$ as $C(S)^{\sigma}_{\zeta,c}\times S^{1}$.   
\end{remark}

\begin{theorem}\label{Lag}
$F:L_{\zeta,c} \rightarrow C(S)$ is a Lagrangian submanifold in $C(S)$. 
\end{theorem}
\begin{proof}
Fix $x_{0}=(p_{0},t_{0}) \in L_{\zeta,c}$. 
For any $X\in T_{p_{0}}C(S)^{\sigma}_{\zeta,c}$, we have
\begin{align}\label{FX}
F_{*}X=(\rho(t_{0})\cdot \tau(t_{0}))_{*}X
\end{align}
and for $\partial/\partial t \in T_{t_{0}}I$ we have
\begin{align}\label{FT}
F_{*}\frac{\partial}{\partial t}=(\rho(t_{0})\cdot \tau(t_{0}))_{*}\biggl(\frac{\dot{\rho}(t_{0})}{\rho(t_{0})}\overrightarrow{p_{0}}+\dot{f}(t_{0})\zeta(p_{0})\biggr). 
\end{align}
By the assumption, $\dot{f}(t_{0})\zeta(p_{0})\neq0$ and it is orthogonal to all tangent vectors on $C(S)^{\sigma}$, 
it follows that $F$ is an immersion. 
Next, it is clear that 
\begin{align*}
& \omega(F_{*}X,F_{*}Y)=\rho^{2}(t_{0})\omega(X,Y)=0, \\
& \omega(F_{*}\partial/\partial t,F_{*}\partial/\partial t)=0\quad and \\
& \omega(F_{*}\partial/\partial t, F_{*}X)=\rho^2(t_{0})\dot{f}(t_{0})\omega(\zeta(p_{0}),X). 
\end{align*}
As mentioned in Remark \ref{abcd}, if two vectors $X$ and $Y$ are tangent to the real form then $\omega(X,Y)=0$
 and note that position vector $\overrightarrow{p_{0}}$ is tangent to the real form. 
Finally, in fact $\omega(\zeta(p_{0}),X)=0$ since 
 $$\omega(\zeta(p_{0}),X)=X(\langle \mu, \zeta\rangle)$$
and by definition of $C(S)^{\sigma}_{\zeta,c}$ the function $\langle \mu, \zeta\rangle$ is a constant $c$ on $C(S)^{\sigma}_{\zeta,c}$. 
Thus we have $F^{*}\omega=0$ and $F$ is a Lagrangian immersion. 
\end{proof}
\section{almost Calabi--Yau manifold}\label{aCY}
In this section, we recall the details about almost Calabi--Yau manifolds, special Lagrangian submanifolds and so on. 
\begin{definition}
Let $(M,\omega)$ be a K\"ahler manifold with complex dimension $m$. If the canonical line bundle $K_{M}$ is trivial, we can take a non-vanishing holomorphic $(m,0)$-form $\Omega$ on $M$. 
Then we call a triple $(M,\omega,\Omega)$ an almost Calabi--Yau manifold. 
Furthermore if the function $\psi:M\rightarrow\mathbb{R}$ defined below is identically constant, we call it a Calabi--Yau manifold. 
\end{definition}

On an almost Calabi--Yau manifold $(M,\omega,\Omega)$, we define a function $\psi$ by
\begin{align*}
e^{2m\psi}\frac{\omega^m}{m!}=(-1)^{\frac{m(m-1)}{2}}\left(\frac{i}{2}\right)^m\Omega\wedge\bar{\Omega}.
\end{align*}

In this section, we always assume that $(M,\omega,\Omega)$ is an almost Calabi--Yau manifold with complex dimension $m$. 
Next, we define the Lagrangian angle of a Lagrangian submanifold. 
\begin{definition}
Let $F:L\rightarrow M$ be a Lagrangian submanifold. 
The Lagrangian angle of $F$ is the map $\theta_{F}:L\rightarrow\mathbb{R}/\pi\mathbb{Z}$ defined by
\begin{equation*}
F^*(\Omega)=e^{i\theta_{F}+mF^*(\psi)}\mathrm dV_{F^*(g)},
\end{equation*}
where $g$ is the Riemannian metric on $M$ with respect to $\omega$. 
\end{definition}
Note that we do not assume that $L$ is oriented. 
Thus $dV_{F^*(g)}$ has ambiguity of the sign. 
Since $F:L\rightarrow M$ is a Lagrangian submanifold, $\theta_{F}$ is well defined. 
For details, see for example Harvey and Lawson \cite[III.1]{HarveyLawson} or Behrndt \cite{Behrndt2}. 
\begin{remark}\label{angle}
Note that $F^{*}\Omega$ is a non-vanishing {\it complex-valued} $m$-form on $L$. 
Hence on each local coordinates $(U, x^{1},\dots,x^{m})$ we can express $F^{*}\Omega$ as 
$$F^{*}\Omega=h(x)dx^{1}\wedge\dots\wedge dx^{m}. $$
Here $h$ is a non-vanishing complex-valued function on $U$. 
Then the Lagrangian angle $\theta_{F}$ is exactly $\arg h$ the argument of $h$. 
\end{remark}

Now we can define special Lagrangian submanifolds. 
\begin{definition}\label{SLag}
Take a constant $\theta \in \mathbb{R}$. 
We say that $F:L\rightarrow M$ is a special Lagrangian submanifold with phase $e^{i\theta}$ if 
the Lagrangian angle $\theta_{F}$ is identically constant $\theta$. 
This condition is equivalent to that
\begin{align*}
F^{*}(\mathrm{Im}(e^{-i\theta}\Omega))=F^{*}(\cos\theta\,\mathrm{Im}\,\Omega-\sin\theta\,\mathrm{Re}\,\Omega)=0. 
\end{align*}
\end{definition}
If $F:L\rightarrow M$ is a special Lagrangian submanifold with phase $e^{i\theta}$, 
then there is a unique orientation on $L$ in which $F^{*}(\mathrm{Re}(e^{-i\theta}\Omega))=F^*(\cos\theta\,\mathrm{Re}\,\Omega+\sin\theta\,\mathrm{Im}\,\Omega)$ is positive.

Historically Harvey and Lawson \cite{HarveyLawson} have defined special Lagrangian submanifolds by calibrations. 
Of course we can define special Lagrangian submanifolds in almost Calabi--Yau manifolds by calibrations as follows. 
Let $g$ be a Riemannian metric with respect to $\omega$. 
Here we define a new Riemannian metric $\tilde{g}$ on $M$ by conformally rescaling by $\tilde{g}=e^{2\psi}g$. 
Then the $m$-form $\mathrm{Re}(e^{-i\theta}\Omega)$ becomes a calibration on the Riemannian manifold $(M,\tilde{g})$ 
and the definition of special Lagrangian submanifolds in $(M,\omega,\Omega)$ is restated as 
a calibrated submanifold in the Riemannian manifold $(M,\tilde{g})$ with respect to $\mathrm{Re}(e^{-i\theta}\Omega)$. 
 
 Here we introduce the generalized mean curvature vector field. 
The generalized mean curvature vector field was introduced by Behrndt in \cite[\S 3]{Behrndt} and later generalized by Smoczyk and Wang in \cite{SmoczykWang}. 

\begin{definition}\label{GenMCF}
The generalized mean curvature vector field $H^{g}$ of $F:L\rightarrow M$ is a normal vector field defined by 
\begin{align*}
H^{g}=H-m(\nabla \psi )^{\bot}. 
\end{align*}
\end{definition}
Here $H$ is the ordinary mean curvature vector field of $F:L\rightarrow M$, $\nabla$ is the gradient with respect to $g$,
 and $\bot$ is the projection from $TM$ to $T^{\bot}L$ it is the $g$-orthogonal complement of $F_{*}(TL)$. 

Note that if $\psi$ is constant or equivalently $(M,\omega,\Omega)$ is Ricci-flat, then $H^{g} \equiv H$. 
As well known, if the ambient space is a Calabi--Yau manifold, then the Lagrangian angle $\theta_{F}$ of a Lagrangian submanifold $F:L\rightarrow M$ and 
its mean curvature vector field $H$ satisfy the equation 
\begin{align*}
H= J\nabla\theta_{F}. 
\end{align*}
More precisely, $H=JF_{*}(\nabla_{F^{*}g}\theta_{F})$ where $\nabla_{F^{*}g}$ is the $(F^{*}g)$-gradient on $L$, however we write it as above for short. 
On the other hand, if the ambient space is an almost Calabi--Yau manifold, the above equation does not hold in general. 
However if we take $H^{g}$ instead of $H$, the above equation holds. This is proved by Behrndt \cite[Prop. 4]{Behrndt}.
\begin{proposition}[cf. {\cite[Prop. 4]{Behrndt}}]\label{GMCF}
Let $F:L\rightarrow M$ be a Lagrangian submanifold in an almost Calabi--Yau manifold. 
Then the generalized mean curvature vector field satisfies $H^{g}= J \nabla\theta_{F}$.
\end{proposition}
It is clear that if $L$ is connected, then L is a special Lagrangian submanifold if and only if $H^{g}\equiv 0$. 
For more motivation to introduce the generalized mean curvature vector field and some properties, refer the paper of Behrndt \cite{Behrndt2}. 

\section{Lagrangian angle}\label{LA}
Let $(C(S),\overline{g})$ be the toric K\"ahler cone over a $(2m-1)$-dimensional toric Sasaki manifold $(S,g)$. 
In this section we assume that the canonical line bundle $K_{C(S)}$ is trivial. 
As mentioned in Remark \ref{trivial}, this assumption is equivalent to that 
there exists an element $\gamma\in(\mathbb{Z}_{\mathfrak{g}})^{*}\cong\mathbb{Z}^{m}$ such that 
$$\langle \gamma, \lambda \rangle = 1$$
for all $\lambda\in\Lambda$. 
Then we can take a non-vanishing holomorphic $(m,0)$-form $\Omega_{\gamma}$ which is expressed as 
$$\Omega_{\gamma}=\exp(\gamma_{1}z^{1}+\dots+\gamma_{m}z^{m})dz^{1}\wedge\dots\wedge dz^{m}$$
on the open dense $T_{\mathbb{C}}^{m}$-orbit $U_{0}\cong(\mathbb{C}^{\times})^{m}$ 
by the logarithmic holomorphic coordinates $(z^{k})_{k=1}^{m}$. 
Thus we have a toric almost Calabi--Yau cone manifold $(C(S),\omega,\Omega_{\gamma})$. 

Remember that in Section \ref{CL} we took the data 
$c \in \mathbb{R}$, 
$\zeta \in \mathfrak{g}$, 
$I \subset \mathbb{R}$, 
$f:I\rightarrow \mathbb{R}$, 
$\rho:I\rightarrow \mathbb{R}^{+}$ and 
$\tau_{0}\in T^{m}$, and we denoted $\tau(t)=\exp(f(t)\zeta)\cdot\tau_{0}$. 
We have defined
a submanifold 
$$C(S)^{\sigma}_{\zeta,c}=\{\, p\in C(S)^{\sigma} \mid \langle\mu(p),\zeta\rangle =c \,\}, $$ an $m$-dimensional manifold 
$$L_{\zeta,c}=C(S)^{\sigma}_{\zeta,c}\times I$$ and a map $F:L_{\zeta,c}\rightarrow C(S)$ by 
$$F(p,t)=\rho(t)\cdot\tau(t)\cdot p.$$ 
Then by Theorem \ref{Lag}, $F:L_{\zeta,c} \rightarrow C(S)$ is a Lagrangian submanifold. 

In this section, we want to compute $F^{*} \Omega_{\gamma}$ and the Lagrangian angle $\theta_{F}$. 
Let $U_{0}\cong (\mathbb{C}^{\times})^{m}$ be an open dense $T^{m}_{\mathbb{C}}$-orbit and 
$(z^{k})_{k=1}^{m}$ be the logarithmic holomorphic coordinates on $U_{0}$. 
Then $C(S)^{\sigma}\cap U_{0}=\{\, (x^{1},\dots,x^{m})\in\mathbb{R}^{m} \,\}$ and 
$$C(S)^{\sigma}_{\zeta,c}\cap U_{0}=\{\, (x^{1},\dots,x^{m}) \mid \langle \mu(x),\zeta\rangle=c \,\}. $$ 
We have only to compute $F^{*} \Omega_{\gamma}$ on this open dense subset. 
If we denote $\tau_{0}=(e^{i\nu^{1}},\dots e^{i\nu^{m}})\in T^{m}$ then we have the following lemma. 

\begin{lemma}\label{angle1234}
The Lagrangian angle of $F:L_{\zeta,c}\rightarrow (C(S),\omega,\Omega_{\gamma})$ is given by
\begin{align}\label{angle123}
\theta_{F}(x,t)=&f(t)\sum_{k=1}^{m}\gamma_{k}\zeta^{k}  +\sum_{k=1}^{m}\gamma_{k}\nu^{k} \\
&+\arg \biggl(   \sum_{k=1}^{m} \biggl(\biggl(\frac{\dot{\rho}(t)}{\rho(t)}\xi^{k}+i\dot{f}(t)\zeta^{k}\biggr)\frac{\partial \langle \mu(x), \zeta \rangle}{\partial x^{k}} \biggr)     \biggr)\quad \mathrm{mod}\, \pi, \notag
\end{align}
where $\xi=(\xi^{1},\dots,\xi^{m})$ is the Reeb field on $C(S)$. 
\end{lemma}
\begin{proof}
Let $\tilde{L}=C(S)^{\sigma}\times I$ and $\iota:L_{\zeta,c}\rightarrow \tilde{L}$ be an inclusion map. 
If we define $\tilde{F}:\tilde{L}\rightarrow C(S)$ by 
$$\tilde{F}(p,t)=\rho(t)\cdot\tau(t)\cdot p, $$
then $F=\tilde{F}\circ \iota$ and $F^{*}\Omega_{\gamma}=\iota^{*}(\tilde{F}^{*}\Omega_{\gamma})$. 
For $\tau=(e^{i\theta^{1}},\dots,e^{i\theta^{m}})\in T^{m}$, the transition map $\tau:U_{0}\rightarrow U_{0}$ is expressed by 
$$\tau\cdot (z^{1},\dots z^{m})=(z^{1}+i\theta^{1},\dots,z^{m}+i\theta^{m}). $$
Since $J(r\frac{\partial}{\partial r})=\xi$ and 
$$\xi=\xi^{1}\frac{\partial}{\partial y^{1}}+\dots+\xi^{m}\frac{\partial}{\partial y^{m}}, $$
we have 
$$r\frac{\partial}{\partial r}=\xi^{1}\frac{\partial}{\partial x^{1}}+\dots+\xi^{m}\frac{\partial}{\partial x^{m}}. $$
Hence for $\rho\in\mathbb{R}^{+}$ the transition map $\rho:U_{0}\rightarrow U_{0}$ is expressed by 
$$\rho\cdot (z^{1},\dots z^{m})=(z^{1}+\xi^{1}\log\rho,\dots,z^{m}+\xi^{m}\log\rho). $$
Then we have 
$$ (\tilde{F}^{*}z^{k})(x^{1},\dots,x^{m},t)= x^{k}+\xi^{k}\log\rho(t)+i(f(t)\zeta^{k}+\nu^{k}). $$
Since 
$$\Omega_{\gamma}=\exp(\gamma_{1}z^{1}+\dots+\gamma_{m}z^{m})dz^{1}\wedge\dots\wedge dz^{m}$$
on $U_{0}$ we have 
$$\tilde{F}^{*}\Omega_{\gamma}=\exp(h_{1}(x,t)+ih_{2}(x,t))d(\tilde{F}^{*}z^{1})\wedge\dots\wedge d(\tilde{F}^{*}z^{m}), $$
where we put 
\begin{align*}
& h_{1}(x,t)=\sum_{k=1}^{m}\gamma_{k}x^{k}+\log\rho(t)\sum_{k=1}^{m}\gamma_{k}\xi^{k}, \\
& h_{2}(x,t)=f(t)\sum_{k=1}^{m}\gamma_{k}\zeta^{k}+\sum_{k=1}^{m}\gamma_{k}\nu^{k}\quad and \\
& d(\tilde{F}^{*}z^{k})=dx^{k}+\biggl( \frac{\dot{\rho}(t)}{\rho(t)}\xi^{k}+i\dot{f}(t)\zeta^{k}\biggr)dt. 
\end{align*}
Fix a point $p_{0} \in C(S)^{\sigma}_{\zeta,c} \cap U_{0}$. 
If we put $\phi(x):=\langle \mu (x), \zeta \rangle -c$, then $C(S)^{\sigma}_{\zeta,c}$ is locally expressed around $p_{0}$ 
as $\{~(x^{1},\dots,x^{m})~|~\phi(x^{1},\dots,x^{m})=0~\}$. 
By the definition of a moment map and the non-degeneracy of K\"ahler form, 
we have $ d \phi = -\omega(\zeta,\cdot) \neq 0 $ at $p_{0}$. 
Hence there exists $k_{0} \in \{1,\dots,m\}$ such that $\frac{\partial \phi}{\partial x^{k_{0}}}(p_{0})\neq 0$. 
Thus by the implicit function theorem, $x^{k_{0}}$ is locally represented as $x^{k_{0}}=x^{k_{0}}(x^{1},\dots,x^{k_{0}-1},x^{k_{0}+1},\dots,x^{m})$. 
Note that since $\phi(x^{1},\dots,x^{m})=0$, we have 
\begin{align*}
\frac{\partial \phi}{\partial x^{\ell}} + \frac{\partial \phi}{\partial x^{k_{0}}} \frac{\partial x^{k_{0}}}{\partial x^{\ell}}=0 
\end{align*}
for all $\ell\neq k_{0}$. 
If we take $(x^{1},\dots,x^{k_{0}-1},x^{k_{0}+1},\dots,x^{m})$ as a local coordinates on $C(S)^{\sigma}_{\zeta,c}$, we have 
\begin{align*}
&\iota^{*}(d(\tilde{F}^{*}z^{1})\wedge\dots\wedge d(\tilde{F}^{*}z^{m}))\\
=&h_{3}(x,t) dx^{1} \wedge \dots \wedge dx^{k_{0}-1} \wedge dx^{k_{0}+1} \wedge \dots \wedge dx^{m} \wedge dt, 
\end{align*}
where 
\begin{align*}
h_{3}(x,t)=(-1)^{m-k_{0}}\biggl( \frac{\partial \langle \mu(x),\zeta \rangle}{\partial x^{k_{0}}} \biggr)^{-1} \biggl(\sum_{\ell=1}^{m}\biggl( \frac{\dot{\rho}(t)}{\rho(t)}\xi^{\ell} + i\dot{f}(t)\zeta^{\ell} \biggr)\frac{\partial \langle \mu(x),\zeta \rangle}{\partial x^{\ell}}\biggr). 
\end{align*}
As mentioned in Remark \ref{angle}, the Lagrangian angle $\theta_{F}$ is 
$$\arg(h_{3}\exp(h_{1}+ih_{2}))=h_{2}+\arg(h_{3}). $$
One can prove that this coincides with the right hand side of the equation (\ref{angle123}). 
\end{proof}
\section{Construction of special Lagrangian submanifolds}\label{constSLag}
Let $(C(S),\omega,\Omega_{\gamma})$ be a toric almost Calabi--Yau cone over a toric Sasaki manifold $(S,g)$. 
In this section, we construct the special Lagrangian submanifolds in $C(S)$. 
Let $F:L(\zeta,c)\rightarrow C(S)$ be a Lagrangian submanifold explained in Section \ref{CL}. 
Then we find the conditions such that $F$ is a special Lagrangian submanifold. 
Remember that we denote the Reeb field $\xi$ and write $\tau_{0}=(e^{i\nu^{1}},\dots,e^{i\nu^{m}})\in T^m$. 
Here we put 
$$N:=\langle \zeta, \gamma \rangle =\sum_{k=1}^{m}\gamma_{k}\zeta^{k}\quad \mathrm{and} \quad \theta:=\sum_{k=1}^{m}\gamma_{k}\nu^{k}. $$
\begin{theorem}\label{54321}
Assume that the function $\rho:I\rightarrow\mathbb{R}^{+}$ is identically constant. 
Take a constant $\theta_{0}\in\mathbb{R}$. 
Then $F:L_{\zeta,c}\rightarrow C(S)$ is a special Lagrangian submanifold with phase $e^{i\theta_{0}}$ if and only if 
\begin{align*}
N=0\quad and \quad \theta+\frac{\pi}{2}=\theta_{0}. 
\end{align*}
\end{theorem}
\begin{proof}
Since $\dot{\rho}(t)=0$, by Lemma \ref{angle1234} we have the Lagrangian angle
$$\theta_{F}(p,t)=f(t)N+\theta+\frac{\pi}{2}. $$
Note that we have assumed that $f(t)$ is not constant. 
Thus the statement follows clearly. 
\end{proof}

\begin{theorem}\label{12345}
We assume that $\zeta=\xi$, and put $\kappa(t):=\log\rho(t)$. 
Take a constant $\theta_{0}\in\mathbb{R}$. 
Then $F:L_{\zeta,c}\rightarrow C(S)$ is a special Lagrangian submanifold with phase $e^{i\theta_{0}}$ 
if and only if
\begin{align}
\mathrm{Im}(e^{i(\theta-\theta_{0})}e^{N(\kappa(t)+if(t))})=\mathrm{const}\label{123456}
\end{align}
\end{theorem}
\begin{proof}
Since $\zeta=\xi$, by Lemma \ref{angle1234}, we have the Lagrangian angle 
\begin{align}
\theta_{F}(p,t)&=f(t)N+\theta+\arg(\dot{\kappa}(t)+i\dot{f}(t))\notag\\
&=\arg(  ( \dot{\kappa}(t)+i\dot{f}(t) )e^{i(f(t)N+\theta)}  ). \label{bca}
\end{align}
Note that $\gamma$ is in $\Delta$ since $\langle \gamma, \lambda \rangle=1$ for all $\lambda \in \Lambda$ and, 
as mentioned in Remark \ref{remarkSasaki}, the Reeb field $\xi=\zeta$ is in $\Delta_{0}^{*}$ 
and this means that $N=\langle \gamma, \zeta \rangle >0$. 
Since the argument of a complex valued function is unchanged by a multiplication of a positive function, 
we can multiply the term in the argument in (\ref{bca}) by $Ne^{N\kappa(t)}$ and we have
\begin{align*}
\theta_{F}(p,t)&=\arg(  ( \dot{\kappa}(t)+i\dot{f}(t) )e^{i(f(t)N+\theta)}  )\\
&=\arg(  N( \dot{\kappa}(t)+i\dot{f}(t) )e^{N\kappa(t)+i(f(t)N+\theta)} ). 
\end{align*}
If we put 
$$h(t)=e^{N\kappa(t)+i(f(t)N+\theta)}, $$
then it is clear that $\theta_{F}(p,t)=\arg(\dot{h}(t))$. 
Thus it follows that $\theta_{F}\equiv \theta_{0}$ constant if and only if 
$$\mathrm{Im}(e^{i(\theta-\theta_{0})}e^{N(\kappa(t)+if(t))})=\mathrm{const}. $$
\end{proof}

\begin{remark}\label{HL}
If we define the curves $c_{j}:I\rightarrow\mathbb{C}^{\times}$ by 
$$c_{j}(t):=\rho^{\xi^{j}}(t)e^{i(f(t)\xi^{j}+\nu^{j})},$$ 
then the equality (\ref{123456}) in Theorem \ref{12345} 
is equivalent to the equality
\begin{align*}
\mathrm{Im}\bigl(e^{-i\theta_{0}}c_{1}^{\gamma_{1}}\cdots c_{m}^{\gamma_{m}}\bigr)=\mathrm{const}. 
\end{align*}
For example in $\mathbb{C}^{m}$, the canonical Reeb field is $\xi=(1,\dots,1)$ and we can take $\gamma=(1,\dots,1)$. 
Then if we take $\theta_{0}=0$ and $\nu^{1}=\dots=\nu^{m}=0$ for example, 
then $c_{1}(t)=\dots=c_{m}(t)$, and we put $c(t):=c_{1}(t)$. 
Then the equality (\ref{123456}) in Theorem \ref{12345} becomes 
$$\mathrm{Im}(c^{m}(t))=\mathrm{const}, $$
and the image of $F:L_{\zeta,c}\rightarrow \mathbb{C}^{m}$ coincides with
$$\{\, (c(t)x^{1},\dots,c(t)x^{m}) \in \mathbb{C}^{m} \mid t\in I,\, x^{j}\in\mathbb{R},\, (x^1)^2+\dots+(x^m)^2=c\,\}. $$ 
Hence this is an extension of examples of special Lagrangian submanifolds mentioned in Theorem 3.5 in Section III.3.B. in the paper of Harvey and Lawson \cite{HarveyLawson}.
\end{remark}

\section{Construction of Lagrangian self-similar solutions}\label{constLss}
Let $(C(S),\omega,\Omega_{\gamma})$ be a toric almost Calabi--Yau cone over a toric Sasaki manifold $(S,g)$. 
Since $C(S)$ has both the cone structure and the almost Calabi--Yau structure, we can consider both the position vector and the generalized mean curvature vector. 
Then we can defined the generalized self-similar solution. 
Let $M$ be a manifold and $F:M\rightarrow C(S)$ be an immersion. 
Then we say that $F$ is a generalized self-similar solution if 
$$H^{g}=\lambda\overrightarrow{F}^{\bot}$$
for some $\lambda\in\mathbb{R}$. 
In this section, we construct the Lagrangian generalized self-similar solutions in $C(S)$. 
Let $F:L_{\zeta,c}\rightarrow C(S)$ be a Lagrangian submanifold explained in Section \ref{CL}. 
Remember that we denote the Reeb field $\xi$ and write $\tau_{0}=(e^{i\nu^{1}},\dots,e^{i\nu^{m}})\in T^m$, and 
in Section \ref{constSLag}, we put 
$$N=\langle \zeta, \gamma \rangle =\sum_{k=1}^{m}\gamma_{k}\zeta^{k}\quad \mathrm{and} \quad \theta=\sum_{k=1}^{m}\gamma_{k}\nu^{k}. $$

\begin{theorem}\label{11223}
Let us assume that $\zeta=\xi$, and put $c(t):=\rho(t)e^{if(t)}\in \mathbb{C}^{\times}$. 
If there exist a function $\theta:I\rightarrow\mathbb{R}/\pi\mathbb{Z}$ and a constant $A\in\mathbb{R}$, and 
$\theta(t)$ and $c(t)$ satisfy the differential equations
\begin{align}\label{LagS2-0}
\begin{cases}
  \dot{c}(t)=e^{i(\theta(t)-\theta)}\overline{c(t)}^{N-1} \\
  \dot{\theta}(t)=A\rho(t)^{N}\sin(f(t)N+\theta-\theta(t)), 
\end{cases}
\end{align}
then $F:L_{\zeta,c} \rightarrow C(S)$ is a Lagrangian generalized self-similar solution with
\begin{align*}
2 c H^{g} = A {\overrightarrow{F}}^{\bot} 
\end{align*}
and Lagrangian angle $\theta_{F}(p,t)=\theta(t)$. 
\end{theorem}

\begin{proof}
First of all, we prove that the Lagrangian angle $\theta_{F}(p,t)$ is equal to $\theta(t)$. 
Since $\zeta=\xi$, by Lemma \ref{angle1234} we have the Lagrangian angle 
$$\theta_{F}(p,t)=f(t)N+\theta+\arg(\dot{\kappa}(t)+i\dot{f}(t)), $$
where $\kappa(t)=\log\rho(t)$. 
Since the argument of a complex valued function is unchanged under the multiplication of a positive real valued function, 
by multiplying $2\rho(t)^2$ we have
\begin{align*}
\arg(\dot{\kappa}(t)+i\dot{f}(t))&=\arg(2\rho(t)^2\dot{\kappa}(t)+2i\rho(t)^2\dot{f}(t))\\
                               &=\arg\biggl(\frac{d}{dt}(\rho(t)^2)+2i\rho(t)^2\dot{f}(t)\biggr). 
\end{align*}
Since $c(t)=\rho(t)e^{if(t)}$, we have
\begin{align*}
\dot{c}(t)=\dot{\rho}(t)e^{if(t)}+i\rho(t)\dot{f}(t)e^{if(t)}
\end{align*}
and multiplying this equation by $2\rho(t)e^{-if(t)}(=2\overline{c(t)})$ we have 
\begin{align}\label{equation1}
2\overline{c(t)}\dot{c}(t)=\frac{d}{dt}(\rho(t)^2)+2i\rho(t)^2\dot{f}(t). 
\end{align}
If we use the differential equation (\ref{LagS2-0}) with respect to $c(t)$ 
then the left hand side of (\ref{equation1}) is equal to 
\begin{align}\label{equation1-1}
2\overline{c(t)}\dot{c}(t)=2e^{i(\theta(t)-\theta)}\overline{c(t)}^N=2\rho(t)^N e^{i(\theta(t)-\theta-f(t)N)}. 
\end{align}
Thus we have
\begin{align*}
\arg(\dot{\kappa}(t)+i\dot{f}(t))=\theta(t)-\theta-f(t)N.
\end{align*}
Consequently we have proved that
\begin{align*}
\theta_{F}(p,t)=\theta(t). 
\end{align*}

We turn to the proof of $2cH^{g}=A {\overrightarrow{F}}^{\bot}$. 
Since $\omega$ is non-degenerate and we have the orthogonal decomposition 
$$T_{F(p)}C(S)=F_{*}(T_{p}L_{\zeta,c})\oplus J(F_{*}(T_{p}L_{\zeta,c}))$$
for all $p$ in $L_{\zeta,c}$, we have only to prove that 
$$\omega(2cH^{g},F_{*}X)=\omega(A{\overrightarrow{F}}^{\bot},F_{*}X)$$
for all $X$ tangent to $L_{\zeta,c}$. 
Furthermore, since $\omega(A{\overrightarrow{F}}^{\bot},F_{*}X)=\omega(A\overrightarrow{F},F_{*}X)$, 
it is equivalent to prove that
$$\omega(2cH^{g},F_{*}X)=\omega(A\overrightarrow{F},F_{*}X). $$
Remember that $L_{\zeta,c}=C(S)^{\sigma}_{\zeta,c}\times I$. 
Fix $x_{0}=(p_{0},t_{0})$ in $L_{\zeta,c}$, $X$ in $T_{p_{0}}C(S)^{\sigma}_{\zeta,c}$ and 
$\partial/\partial t$ in $T_{t_{0}}I$. 
See the equalities (\ref{FX}) and (\ref{FT}) in the proof of Theorem \ref{Lag}, we have 
\begin{align*}
& F_{*}X=(\rho(t_{0})\cdot \tau(t_{0}))_{*}X \\
& F_{*}\frac{\partial}{\partial t}=(\rho(t_{0})\cdot \tau(t_{0}))_{*}\biggl(\frac{\dot{\rho}(t_{0})}{\rho(t_{0})}\overrightarrow{p_{0}}+\dot{f}(t_{0})\xi(p_{0})\biggr). 
\end{align*} 
By Proposition \ref{GMCF} we have 
\begin{align*}
H^{g}=JF_{*}(\nabla_{F^{*}g}\theta_{F}), 
\end{align*}
where $\nabla_{F^{*}g}$ is the $(F^{*}g)$-gradient on $L$. 
By the definition of the position vector, one can prove that 
\begin{align*}
\overrightarrow{F}(x_{0})=(\rho(t_{0})\cdot \tau(t_{0}))_{*}(\overrightarrow{p_{0}})
\end{align*} 
at $x_{0}=(p_{0},t_{0})$. 
Note that we have proved that the Lagrangian angle
\begin{align*}
\theta_{F}(p,t)=\theta(t)
\end{align*}
and this function is independent of any points in $C(S)^{\sigma}_{\zeta,c}$. 
Thus if $X$ is tangent to $C(S)^{\sigma}_{\zeta,c}$ at $p_{0}$, then we have
\begin{align*}
\omega(2cH^{g},F_{*}X)&=2c\, \omega(JF_{*}(\nabla_{F^{*}g}\theta_{F}), F_{*}X)=-2c(F^{*}g)(\nabla_{F^{*}g}\theta_{F},X)\\
&=-2cX(\theta_{F})=0. 
\end{align*}
Since if we substitute two vectors tangent to the real form into $\omega$ then it is zero, and $\overrightarrow{p_{0}}$ is tangent to the real form, 
for $X$ tangent to $C(S)^{\sigma}_{\zeta,c}$ at $p_{0}$ we have
\begin{align*}
\omega(A\overrightarrow{F},F_{*}X)=A\rho^{2}(t_{0})\omega(\overrightarrow{p_{0}},X)=0. 
\end{align*}
Thus we have 
\begin{align*}
\omega(2cH^{g},F_{*}X)=0=\omega(A\overrightarrow{F},F_{*}X)
\end{align*}
for all $X$ tangent to $C(S)^{\sigma}_{\zeta,c}$ at $p_{0}$. 
Next, for $\partial/\partial t$ tangent to $I$ at $t_{0}$, we have
\begin{align*}
\omega(2cH^{g}, F_{*} \frac{\partial}{\partial t})&=2c\, \omega(JF_{*}(\nabla_{F^{*}g}\theta_{F}),F_{*}\frac{\partial}{\partial t})=-2c(F^{*}g)(\nabla_{F^{*}g}\theta_{F},\frac{\partial}{\partial t})\\
&=-2c\frac{\partial}{\partial t}\theta_{F}=-2c\dot{\theta}(t_{0})\\
&=-2cA\rho(t_{0})^{N}\sin(f(t_{0})N+\theta-\theta(t_{0})). 
\end{align*}
In the last equality, we use the differential equation (\ref{LagS2-0}) with respect to $\theta(t)$. 
On the other hand, we have 
\begin{align*}
\omega(A\overrightarrow{F},F_{*}\frac{\partial}{\partial t})&=A\rho^{2}(t_{0})\dot{f}(t_{0})\omega(\overrightarrow{p_{0}},\xi(p_{0}))=A\rho^{2}(t_{0})\dot{f}(t_{0})\overrightarrow{p_{0}}(\langle \mu, \xi \rangle))\\
 &=A\rho^{2}(t_{0})\dot{f}(t_{0})\frac{d}{d\rho}\biggl|_{\rho=1}\langle \mu(\rho\cdot p_{0}), \xi \rangle \\
 &=A\rho^{2}(t_{0})\dot{f}(t_{0})\frac{d}{d\rho}\biggl|_{\rho=1}\rho^2 \langle \mu( p_{0}), \xi \rangle \\
 &=2cA\rho^{2}(t_{0})\dot{f}(t_{0}). 
 \end{align*}
In the fourth equality, we use $\langle \mu(\rho\cdot p_{0}), \xi \rangle=\rho^2 \langle \mu( p_{0}), \xi \rangle$ for a $\rho\in\mathbb{R}^{+}$ action and 
it follows by the definition of the moment map (\ref{momentmap}). 
In the last equality, remember that for $p_{0}$ in $C(S)^{\sigma}_{\zeta,c}$ (now $\zeta=\xi$ by the assumption) $\langle \mu(p_{0}) , \zeta \rangle =c $ by the definition of $C(S)^{\sigma}_{\zeta,c}$. 
By the equality (\ref{equation1}), we know that $2\rho^{2}(t_{0})\dot{f}(t_{0})$ is the imaginary part of $2\overline{c(t_{0})}\dot{c}(t_{0})$, 
and using the equality (\ref{equation1-1}) we show that 
$$2\rho^{2}(t_{0})\dot{f}(t_{0})=2\rho^{N}(t_{0})\sin(\theta(t_{0})-\theta-f(t_{0})N)$$
Thus we have 
\begin{align*}
\omega(2cH^{g}, F_{*} \frac{\partial}{\partial t})=\omega(A\overrightarrow{F},F_{*}\frac{\partial}{\partial t}). 
\end{align*}
This means that $2cH^{g}=A {\overrightarrow{F}}^{\bot}$. 
\end{proof}

\begin{remark}\label{JLT123}
Here we assume that all $\xi^{j}\neq0$. 
If we define curves $c_{j}:I\rightarrow\mathbb{C}^{*}$ by 
$$c_{j}(t):=\rho^{\xi^{j}}(t)e^{i(f(t)\xi^{j}+\nu^{j})}, $$
then the differential equations (\ref{LagS2-0}) in Theorem \ref{11223} 
are equivalent to the following differential equations. 
\begin{align}\label{LagS1-0}
\begin{cases}
 \frac{d}{dt}c_{j}^{1/\xi^{j}}(t)=e^{i\theta(t)}\overline{c_{1}^{\gamma_{1}}(t) \cdots c_{j}^{\gamma_{j}-1/\xi^{j}}(t) \cdots  c_{m}^{\gamma_{m}}(t)} \quad (j=1,\dots,m)\\
 \frac{d}{dt}\theta(t)=A\, \mathrm{Im}(e^{-i\theta(t)}c_{1}^{\gamma_{1}}(t) \cdots c_{m}^{\gamma_{m}}(t)). 
\end{cases}
\end{align}
For example in $\mathbb{C}^{m}$, the canonical Reeb field is $\xi=(1,\dots,1)$ and $\gamma=(1,\dots,1)$. 
Then if we take $\theta_{0}=0$ and $\nu^{1}=\dots=\nu^{m}=0$ for example, 
then the above equality (\ref{LagS1-0}) becomes 
\begin{align*}
\begin{cases}
 \frac{d}{dt}c_{j}(t)=e^{i\theta(t)}\overline{c_{1}(t) \cdots c_{j-1}(t)\cdot c_{j+1}(t) \cdots  c_{m}(t)} \quad (j=1,\dots,m)\\
 \frac{d}{dt}\theta(t)=A\, \mathrm{Im}(e^{-i\theta(t)}c_{1}(t) \cdots c_{m}(t)), 
\end{cases}
\end{align*}
and the image of $F:L_{\zeta,c}\rightarrow \mathbb{C}^{m}$ coincides with
$$\{\, (c_{1}(t)x^{1},\dots,c_{m}(t)x^{m}) \in \mathbb{C}^{m} \mid t\in I,\, x^{j}\in\mathbb{R},\, (x^1)^2+\dots+(x^m)^2=c\,\}. $$ 
This differential equations appear in Theorem A in the paper of Joyce, Lee and Tsui \cite{JoyceLeeTsui}. 
Hence this is one of extension of the paper of Joyce, Lee and Tsui in $\mathbb{C}^{m}$ to the toric almost Calabi--Yau cone.
\end{remark}

\section{Examples}\label{TKC}
In this section, we apply the theorems and construct some concrete examples of special Lagrangians and Lagrangian self-similar solutions. 
As explained in Remark \ref{remarkSasaki} in Section \ref{TS}, the moment image of a toric K\"ahler cone is a strongly convex good rational polyhedral cone. 
Conversely, we can construct a toric K\"ahler cone from a strongly convex good rational polyhedral cone by the Delzant construction. 

Let 
$$\Delta = \{\, y \in \mathfrak g^{\ast} \mid \langle y, \lambda_i \rangle \geq 0\ \mathrm{for\ }\ i = 1, \cdots, d \,\}-\{0\}$$
be a strongly convex good rational polyhedral cone and put the (open) dual cone
$$\Delta_{0}^{*}=\{\, \xi \in \mathfrak{g} \mid \langle v,\xi\rangle >0\ \mathrm{for\ all}\ v \in \Delta \,\}. $$
\begin{proposition}\label{Delzant}
For $\Delta$ and $\xi \in \Delta_{0}^{*}$, there exists a compact connected toric Sasaki manifold $(S,g)$ whose moment image is 
equal to $\Delta$ and whose Reeb vector field is generated by $\xi$. 
\end{proposition}
This proposition is proved by the Delzant construction, for details see \cite{Lerman} and \cite{MartelliSparksYau2}. 
Of course the cone $(C(S),\overline{g})$ of $(S,g)$ is a toric K\"ahler manifold whose moment image is equal to $\Delta$. 

As mentioned in Remark \ref{trivial} in Section \ref{TS}, the canonical line bundle $K_{C(S)}$ is trivial if and only if 
there exists an element $\gamma$ in $(\mathbb{Z}_{\mathfrak{g}})^{*}\cong\mathbb{Z}^{m}$ such that 
$\langle \gamma, \lambda_{j} \rangle=1$ for all $j=1,\dots, d$, 
and using $\gamma$ we can construct a non-vanishing holomorphic $(m,0)$-form $\Omega_{\gamma}$ that is written by  
\begin{align}\label{Omega}
\Omega_{\gamma}=\exp(\gamma_{1}z^{1}+\dots\gamma_{m}z^{m})dz^{1}\wedge\dots\wedge dz^{m}
\end{align}
on an open dense $T^{m}_{\mathbb{C}}$-orbit by the logarithmic holomorphic coordinates. 
This condition is called the {\it height 1} and in fact there exists a definition of the {\it height $\ell$} for some $\ell\in\mathbb{Z}$, for example see Cho-Futaki-Ono \cite{ChoFutakiOno}. 
Here we want to introduce the results in \cite{ChoFutakiOno}.  
\begin{theorem}[cf. Theorem 1.2 in \cite{ChoFutakiOno}]\label{CFO1}
 Let $S$ be a compact toric Sasaki manifold
 with $c_{1}^{B} > 0$ and $c_{1}(D) = 0$. Then by
deforming the Sasaki structure varying  the Reeb vector field, we obtain a Sasaki-Einstein structure.
\end{theorem}
 We do not explain the meanings of $c_{1}^{B} $ and $c_{1}(D)$ in this paper, but in \cite{ChoFutakiOno} it is proved that 
 the condition with $c_{1}^{B} > 0$ and $c_{1}(D) = 0$ is equivalent to the {\it height $\ell$} for some $\ell\in\mathbb{Z}$. 
 Note that $(S,g)$ is Sasaki-Einstein if and only if $(C(S),\omega)$ is Ricci flat. 
 Thus, if we use Theorem \ref{CFO1}, then we get a toric Calabi--Yau cone $(C(S),\omega,\Omega_{\gamma})$ rather than {\it almost} Calabi--Yau . 
 The merit of using the toric Calabi--Yau is that $H^{g}$ coincides with $H$. 
 
From now on, we restrict ourselves to the case of $\dim_{\mathbb{C}}C(S)=3$. 
There is a useful proposition (c.f. \cite{ChoFutakiOno} ) to check whether given inward conormal vectors $\lambda_{i}$ satisfy the goodness condition (\ref{goodcondition}) of Definition \ref{good}. 
\begin{proposition}\label{gooddiagram}
Let $\Delta$ be a strongly convex rational polyhedral cone in $\mathbb{R}^3$ given by
$$ \Delta= \{\, y \in \mathbb{R}^3 \mid  \langle y, \lambda_i\rangle \geq 0,\ j = 1, \cdots, d\, \}-\{0\}$$
$$ \lambda_1 = \left(\begin{array}{c} 1 \\ p_1 \\ q_1\end{array}\right), \cdots,
\lambda_d = \left(\begin{array}{c} 1 \\ p_d \\ q_d\end{array}\right).$$
Then $\Delta$ is good in the sense of Definition \ref{good} if and only if
either
\begin{enumerate}
\item$|p_{i+1}- p_i | = 1 $ or
\item$|q_{i+1} - q_i| = 1$ or
\item$p_{i+1}- p_i $ and $q_{i+1} - q_i$ are relatively prime non-zero integers 
\end{enumerate}
for $i = 1,\ \cdots, d$ where we have put $\lambda_{d+1} = \lambda_1$.
\end{proposition}

\begin{example}\label{ex1}
Take an integer $g\geq 1$. 
If $g=1$, let $\Delta$ be the strongly convex rational polyhedral cone defined by 
\begin{align*}
\Delta=\Delta_{1}=\{\, y \in \mathbb{R}^3 \mid \langle y , \lambda_{i}\rangle \geq 0,\, i=1,2,3,4\,\}-\{0\}
\end{align*}
with 
\begin{align*}
\lambda_{1}:=\begin{pmatrix} 1 \\ -1 \\ -1 \end{pmatrix},~
\lambda_{2}:=\begin{pmatrix} 1 \\ 0 \\ -1 \end{pmatrix},~
\lambda_{3}:=\begin{pmatrix} 1 \\ 1 \\ 0 \end{pmatrix},~
\lambda_{4}:=\begin{pmatrix} 1 \\ 2 \\ 3 \end{pmatrix}. 
\end{align*}
If $g\geqq2$ 
let $\Delta$ be the strongly convex rational polyhedral cone defined by 
\begin{align*}
\Delta=\Delta_{g}=\{\, y \in \mathbb{R}^3 \mid \langle y , \lambda_{i}\rangle \geq 0,~i=1,\dots,g+3\,\}-\{0\}
\end{align*}
with 
\begin{align*}
\lambda_{1}:=\begin{pmatrix} 1 \\ -1 \\ -1 \end{pmatrix},~
\lambda_{k}:=\begin{pmatrix} 1 \\ k-2 \\ (k-2)^2-1 \end{pmatrix}~(k=2,3,\dots,g+2),~
\lambda_{g+3}:=\begin{pmatrix} 1 \\ -2 \\ g^2 \end{pmatrix},~
\end{align*}
Then by Proposition \ref{gooddiagram}, $\Delta$ is a strongly convex {\it good} rational polyhedral cone.
Since we can take $\gamma$ as $(1,0,0)$ so that $\langle \gamma, \lambda_{j} \rangle =1$ for $j=1,\dots , g+3$, 
this condition satisfies the {\it height 1} and we can use Theorem \ref{CFO1}. 
Let $(C(S),\omega)$ be a toric K\"ahler manifold whose moment image is equal to $\Delta$. 
The existence of it is guaranteed by Proposition \ref{Delzant}. 
If necessary, we deform the K\"ahler form $\omega$ and Reeb field $\xi$ on $C(S)$ so that $(C(S),\omega)$ is Ricci flat by Theorem \ref{CFO1}. 
Thus we can assume that $(C(S),\omega)$ is Ricci flat. 
Furthermore, since we can take $\gamma$ as above, the canonical line bundle $K_{C(S)}$ is trivial and we have a 
non-vanishing holomorphic $(3,0)$-form $\Omega_{\gamma}$ on $C(S)$. 
Thus we have a Calabi--Yau cone $M_{g}=(C(S),\omega,\Omega_{\gamma})$ and denote its Reeb field by $\xi$. 

For example, if we take 
$$c:=\frac{1}{2}\langle \gamma,\xi\rangle \quad \mathrm{and} \quad \zeta:=\xi, $$ 
then $\zeta$ and $c$ satisfy the assumptions (\ref{ass1}) and (\ref{ass2}) in Section \ref{CL}, 
which proved in Proposition \ref{app1} in Appendix \ref{app}. 
Then the shape of $\Delta_{\zeta,c}=\Delta\cap H_{\zeta,c}$ is a $(g+3)$-gon, 
which proved in Proposition \ref{app2} in Appendix \ref{app}. 
For example if $g=1$ then $\Delta_{\zeta,c}$ is a quadrilateral and if $g=2$ then $\Delta_{\zeta,c}$ is a pentagon. 

Remember that $\mu^{\sigma}$, the restriction of the moment map $\mu$ to the real form $C(S)^{\sigma}$, is a $2^{3}(=8)$-fold covering of $\Delta$, 
and we have defined $C(S)^{\sigma}_{\zeta,c}=(\mu^{\sigma})^{-1}(\Delta_{\zeta,c})$. 
Hence the topological shape of the $C(S)^{\sigma}_{\zeta,c}$ is a 2-dimensional surface constructed from $8$-copies of $\Delta_{\zeta,c}$ that is glued with certain boundaries. 
In this setting, we can see that 
$$C(S)^{\sigma}_{\zeta,c}\cong\Sigma_{g}, $$
where $\Sigma_{g}$ is a closed surface of genus $g$. 
This will be explained in Proposition \ref{app3} in Appendix \ref{app}.

{\bf Special Lagrangian. }
First we construct special Lagrangian submanifolds using Theorem \ref{12345}.  
Now $N=\langle \gamma, \zeta \rangle>0$. 
For example take $\theta_{0}=0$.  
Then, for example, take an open interval $I=(0,\pi)$, and define $f:I\rightarrow\mathbb{R}$ and $\rho:I\rightarrow\mathbb{R}^{+}$ by 
$$f(t)=\frac{1}{N}t\quad \mathrm{and} \quad \rho(t)=\biggl(\frac{1}{\sin t}\biggr)^{1/N}, $$
and take $\tau_{0}=(e^{i\nu^{1}},e^{i\nu^{2}},e^{i\nu^{3}})$ in $T^3$ as $\nu^{1}=\nu^{2}=\nu^{3}=0$. 
Then $\theta=\gamma_{1}\nu^{1}+\gamma_{2}\nu^{2}+\gamma_{3}\nu^{3}=0$. 
This setting satisfies the equality (\ref{123456}). 
Thus $F:L_{\zeta,c}\rightarrow M_{g}$ is a special Lagrangian submanifold and $L_{\zeta,c}$ is diffeomorphic to
$$L_{\zeta,c}\cong\Sigma_{g}\times\mathbb{R}. $$
Note that of course the map $F$ and $L_{\zeta,c}$ depend on $g$, 
and in Example \ref{ex-intro} we denote these by $F^{1}_{g}:L^{1}_{g}\rightarrow M_{g}$. 

{\bf Lagrangian self-similar solution. }
Next we construct Lagrangian (generalized) self-similar solutions using Theorem \ref{11223}. 
Now $N=\langle \gamma, \zeta \rangle>0$. 
For example take 
$$\theta(t)=Nt+\frac{\pi}{2}\quad \mathrm{and} \quad A=-N. $$  
Then, for example, take an interval $I=\mathbb{R}$, and define $f:I\rightarrow\mathbb{R}$ and $\rho:I\rightarrow\mathbb{R}^{+}$ by 
$$f(t)=t\quad \mathrm{and} \quad \rho(t)=1, $$
and take $\tau_{0}=(e^{i\nu^{1}},e^{i\nu^{2}},e^{i\nu^{3}})$ in $T^3$ as $\nu^{1}=\nu^{2}=\nu^{3}=0$. 
Then $\theta=\gamma_{1}\nu^{1}+\gamma_{2}\nu^{2}+\gamma_{3}\nu^{3}=0$. 
This setting satisfies the differential equations (\ref{LagS2-0}). 
Thus $F:L_{\zeta,c}\rightarrow C(S)$ is a Lagrangian self-similar solution (self-shrinker). 
Furthermore as mentioned in Remark \ref{S1}, we can reduce $I$ to $S^1$, hence we have 
a {\it compact} Lagrangian self-shrinker $F:L_{\zeta,c}\rightarrow M_{g}$ with 
$$H^{g}=-{\overrightarrow{F}}^{\bot}$$
which is diffeomorphic to 
$$L_{\zeta,c}\cong\Sigma_{g}\times S^{1}. $$
Note that of course the map $F$ and $L_{\zeta,c}$ depend on $g$, 
and in Example \ref{ex-intro} we denote these by $F^{2}_{g}:L^{2}_{g}\rightarrow M_{g}$. 
\end{example}

\begin{remark}
In $M_{g}(=C(S))$ constructed above, it is clear that the real form $C(S)^{\sigma}$ itself is one of the most typical examples of special Lagrangian submanifold in $C(S)$, and it is a cone. 
Hence $C(S)^{\sigma}$ is also diffeomorphic to $\Sigma_{g}\times\mathbb{R}$. 
However the above example $F^{1}_{g}:L^{1}_{g}\rightarrow M_{g}$ is different from the real form itself, especially it dose not have a cone shape. 
\end{remark}

\appendix
\section{}\label{app}
In this appendix, we give some proofs for the statements mentioned in Example \ref{ex1} in Section \ref{TKC}. 
\begin{proposition}\label{app1}
$\zeta$ and $c$ in Example \ref{ex1} satisfy the assumptions (\ref{ass1}) and (\ref{ass2}) in Section \ref{CL}. 
\end{proposition}
\begin{proof}
First, it is clear that $\frac{1}{2}\gamma$ is in $\mathrm{Int}\, \Delta$ and it is also in $H_{\zeta,c}$. 
This proves that $\zeta$ and $c$ satisfy the assumption (\ref{ass1}). 
Next we prove that $\zeta$ and $c$ satisfy the assumption (\ref{ass2}) by the proof of contradiction. 
Assume that there exists $y$ in $\Delta\cap H_{\zeta,c}$ such that $\zeta$ is in $\mathfrak{z}_{y}$.  
Here remember that 
$$\mathfrak{z}_{y}=\mathrm{Span}_{\mathbb{R}}\{\, \lambda_{j} \mid \langle y,\lambda_{j}\rangle=0 \,\}. $$
Since $y$ is in $\Delta$ and, as mentioned in Remark \ref{remarkSasaki}, the Reeb field $\xi$ is in $\Delta_{0}^{*}$, 
this means that $\langle y , \zeta \rangle =\langle y , \xi \rangle>0$. 
On the other hand, the pairing of $y$ and all elements in $\mathfrak{z}_{y}$ is zero. 
This is in contradiction to that $\zeta$ is in $\mathfrak{z}_{y}$.  
Thus we have proved that $\zeta$ and $c$ satisfy the assumption (\ref{ass2}). 
\end{proof}

\begin{proposition}\label{app2}
The shape of $\Delta_{\zeta,c}=\Delta\cap H_{\zeta,c}$ in Example \ref{ex1} is a $(g+3)$-gon. 
\end{proposition}
\begin{proof}
First, we denote the facet of $\Delta$ defined by $\lambda_{j}$ by 
$$F_{j}=\{\, y \in \Delta \mid \langle y , \lambda_{j}\rangle =0 \,\}$$
for $j=1,\dots,g+3$. 
Next, take an element $y$ in $F_{j}$ and put 
$\kappa:=\frac{c}{\langle y, \zeta \rangle}.$
Since $\frac{1}{2}\gamma$ and $y$ are in $\Delta$ and $\zeta=\xi$ is in $\Delta_{0}^{*}$, it follows that 
$c=\frac{1}{2}\langle \gamma,\xi\rangle>0$, $\langle y, \zeta \rangle >0$ and $\kappa>0$.  
Then $\kappa y$ is in $F_{j}$ and $H_{\zeta,c}$. 
This means that the hyperplane $H_{\zeta,c}$ intersects all facets of $\Delta$. 
Thus we have proved that $\Delta_{\zeta,c}$ is a $(g+3)$-gon. 
\end{proof}

\begin{proposition}\label{app3}
Under the setting in Example \ref{ex1}, 
$$C(S)^{\sigma}_{\zeta,c}\cong\Sigma_{g}, $$
where $\Sigma_{g}$ is a closed surface of genus $g$. 
\end{proposition}
\begin{proof}
There exists an open dense $T_{\mathbb{C}}^{3}$-orbit on $C(S)$. 
We identify $T_{\mathbb{C}}^{3}$ with $(\mathbb{C}^{\times})^{3}$. 
It is clear that the real form of $(\mathbb{C}^{\times})^{3}$ is $(\mathbb{R}^{\times})^{3}$ 
and it has $8$ connected components $\mathbb{R}^{3}(\kappa_{1}, \kappa_{2}, \kappa_{3})$, 
where $\kappa_{i}$ are $+1$ or $-1$ and we define
$$\mathbb{R}^{3}(\kappa_{1}, \kappa_{2}, \kappa_{3})=\{\, (x_{1},x_{2},x_{3})\in\mathbb{R}^{3} \mid \kappa_{1}x_{1}>0, \kappa_{2}x_{2}>0, \kappa_{3}x_{3}>0 \,\}. $$
There is a standard diffeomorphism from each $\mathbb{R}^{3}(\kappa_{1}, \kappa_{2}, \kappa_{3})$ to $\mathbb{R}^{3}$ defined by 
$$-\log |\cdot| :\mathbb{R}^{3}(\kappa_{1}, \kappa_{2}, \kappa_{3})\rightarrow \mathbb{R}^{3}, $$
that is , $(x_{1},x_{2},x_{3})$ maps to $(-\log|x_{1}|,-\log|x_{2}|,-\log|x_{3}|)$. 
In the algebraic toric geometry, there is a concept of manifolds with corner associated with toric varieties. 
From this view point, we can consider that $\mathbb{R}^{3}$ is rescaled and embedded into $\Delta$, that is a manifold with corner. 
This means that the infinity toward the direction of $\lambda_{j}$ in $\mathbb{R}^{3}$ corresponds to the facet $F_{j}$ of $\Delta$ defined by $\lambda_{j}$.
For more general treatment, see Oda \cite{Oda}. 
In this sense, we identify $\mathbb{R}^{3}$ and $\mathrm{Int}\ \Delta$, and we identify the infinity toward the direction of $\lambda_{j}$ in $\mathbb{R}^{3}$ and the facet $F_{j}$ of $\Delta$ defined by $\lambda_{j}$.
For each inward conormal $\lambda_{j}=(\lambda^{1}_{j},\lambda^{2}_{j},\lambda^{3}_{j})$ of $\Delta$, then consider a curve $c_{j}(t)$ in $\mathbb{R}^{3}\cong \mathrm{Int}\, \Delta$ defined by 
$$c_{j}(t)=t\lambda_{j}=(\lambda^{1}_{j}t,\lambda^{2}_{j}t,\lambda^{3}_{j}t). $$
Then the pull back of $c_{j}(t)$ to $\mathbb{R}^{3}(\kappa_{1}, \kappa_{2}, \kappa_{3})$ by $-\log |\cdot|$ is 
$$\tilde{c}_{j}(t)=(\kappa_{1}e^{-\lambda^{1}_{j}t},\kappa_{2}e^{-\lambda^{2}_{j}t},\kappa_{3}e^{-\lambda^{3}_{j}t})$$
and if we put $s=e^{-t}>0$ then this curve $\tilde{c}_{j}(t)$ in $\mathbb{R}^{3}(\kappa_{1}, \kappa_{2}, \kappa_{3})$ is written by 
$$\tilde{c}_{j}(s)=(\kappa_{1}s^{\lambda^{1}_{j}},\kappa_{2}s^{\lambda^{2}_{j}},\kappa_{3}s^{\lambda^{3}_{j}}). $$
If this curve tends to the facet $F_{j}$, then it is equivalent to $t\rightarrow +\infty$ and also $s\rightarrow +0$. 
If we allow to take $s=0$, then the point $\tilde{c}_{j}(0)$ can be considered as in the facet $F_{j}$ and furthermore 
if we allow to take $s<0$, then the curve $\tilde{c}_{j}(s)$ is in 
$$\mathbb{R}^{3}((-1)^{\lambda^{1}_{j}}\kappa_{1},  (-1)^{\lambda^{2}_{j}}\kappa_{2}, (-1)^{\lambda^{3}_{j}}\kappa_{3}). $$
This means that if we prepare $8$ copies of $\Delta$ and give the labels formally to each $\Delta$ as 
\begin{align}\label{domains1}
\begin{matrix}
\Delta(+1,+1,+1),&\Delta(+1,+1,-1),&\Delta(+1,-1,+1),&\Delta(+1,-1,-1), \\
\Delta(-1,+1,+1),&\Delta(-1,+1,-1),&\Delta(-1,-1,+1),&\Delta(-1,-1,-1),
\end{matrix}
\end{align}
then $\Delta(\kappa_{1},\kappa_{2},\kappa_{3})$ and $\Delta((-1)^{\lambda^{1}_{j}}\kappa_{1},  (-1)^{\lambda^{2}_{j}}\kappa_{2}, (-1)^{\lambda^{3}_{j}}\kappa_{3})$ are glued together along 
the facet $F_{j}$ defined by $\lambda_{j}$. 

In the above observation, we consider the gluing relation of $8$ copies of $\Delta$ however, the glueing relation of $\Delta_{\zeta,c}$ is the same as $\Delta$. 
That is, if we prepare $8$ copies of $\Delta_{\zeta,c}$ and give the labels formally to each $\Delta_{\zeta,c}$ as same as (\ref{domains1}), 
then $\Delta_{\zeta,c}(\kappa_{1},\kappa_{2},\kappa_{3})$ and $\Delta_{\zeta,c}((-1)^{\lambda^{1}_{j}}\kappa_{1},  (-1)^{\lambda^{2}_{j}}\kappa_{2}, (-1)^{\lambda^{3}_{j}}\kappa_{3})$ 
are glued together along the edge $E_{j}=F_{j}\cap\Delta_{\zeta,c}$ defined by $\lambda_{j}$. This is the topological shape of $C(S)^{\sigma}_{\zeta,c}$. 

Then one can check that $C(S)^{\sigma}_{\zeta,c}\cong\Sigma_{g}$ by the straight forward observations glueing $8$ copies of $\Delta_{\zeta,c}$ as above relations. 
In Figure 1 and Figure 2, we draw the image of the way of gluing in the case $g=1$ and $g=2$ respectively. 
In these figures, we write $\Delta_{\zeta,c}(\kappa_{1},\kappa_{2},\kappa_{3})$ by $(\kappa_{1},\kappa_{2},\kappa_{3})$ for short and the edge $E_{j}$ by $j$ for short, and glue same labels together. 
Note that in Figure 2 we write a pentagon as a quadrilateral by joining edge 4 and edge 5 flatly to write a picture easily. 
\end{proof}

  \begin{figure}[h]
  \begin{center}
  \setlength\unitlength{1truecm}
  \begin{picture}(11.5,4.5)(0,0)

  
  \put(1.75,2.5){\vector(1,0){2}}
  \put(5.75,2.5){\vector(-1,0){2}}
  \put(5.75,2.5){\vector(1,0){2}}
  \put(9.75,2.5){\vector(-1,0){2}}
  
  \put(3.75,0.5){\vector(-1,0){2}}
  \put(3.75,0.5){\vector(1,0){2}}
  \put(7.75,0.5){\vector(-1,0){2}}
  \put(7.75,0.5){\vector(1,0){2}}
  
  \put(3.75,4.5){\vector(-1,0){2}}
  \put(3.75,4.5){\vector(1,0){2}}
  \put(7.75,4.5){\vector(-1,0){2}}
  \put(7.75,4.5){\vector(1,0){2}}

  \put(1.75,4.5){\vector(0,-1){2}}
  \put(3.75,2.5){\vector(0,1){2}}
  \put(5.75,4.5){\vector(0,-1){2}}
  \put(7.75,2.5){\vector(0,1){2}}
  \put(9.75,4.5){\vector(0,-1){2}}
  
  \put(1.75,0.5){\vector(0,1){2}}
  \put(3.75,2.5){\vector(0,-1){2}}
  \put(5.75,0.5){\vector(0,1){2}}
  \put(7.75,2.5){\vector(0,-1){2}}
  \put(9.75,0.5){\vector(0,1){2}}

  \put(2.05,3.5){$(+,+,+)$} 
  \put(4.05,3.5){$(-,-,+)$}
  \put(6.05,3.5){$(+,+,-)$}
  \put(8.05,3.5){$(-,-,-)$}  
  
  \put(2.05,1.5){$(-,+,-)$} 
  \put(4.05,1.5){$(+,-,-)$} 
  \put(6.05,1.5){$(-,+,+)$} 
  \put(8.05,1.5){$(+,-,+)$}

  \put(1.77,3.5){$1$} 
  \put(5.55,3.5){$1$} 
  \put(5.77,3.5){$1$} 
  \put(9.55,3.5){$1$} 
  \put(1.77,1.5){$1$} 
  \put(5.55,1.5){$1$} 
  \put(5.77,1.5){$1$} 
  \put(9.55,1.5){$1$}

  \put(2.6,2.55){$2$} 
  \put(4.6,2.55){$2$} 
  \put(6.6,2.55){$2$} 
  \put(8.6,2.55){$2$} 
  \put(2.6,2.2){$2$} 
  \put(4.6,2.2){$2$} 
  \put(6.6,2.2){$2$} 
  \put(8.6,2.2){$2$}

  \put(3.55,3.5){$3$} 
  \put(3.77,3.5){$3$} 
  \put(7.55,3.5){$3$} 
  \put(7.77,3.5){$3$} 
  \put(3.55,1.5){$3$} 
  \put(3.77,1.5){$3$} 
  \put(7.55,1.5){$3$} 
  \put(7.77,1.5){$3$}

  \put(2.6,4.2){$4$} 
  \put(4.6,4.2){$4$} 
  \put(6.6,4.2){$4$} 
  \put(8.6,4.2){$4$} 
  \put(2.6,0.55){$4$} 
  \put(4.6,0.55){$4$} 
  \put(6.6,0.55){$4$} 
  \put(8.6,0.55){$4$}

  \put(1.35,3.5){$1_{a}$} 
  \put(9.82,3.5){$1_{a}$} 
  
  \put(1.35,1.5){$1_{b}$} 
  \put(9.82,1.5){$1_{b}$}

  \put(2.6,4.6){$4_{a}$} 
  \put(2.6,0.2){$4_{a}$} 
  
  \put(4.6,4.6){$4_{b}$} 
  \put(4.6,0.2){$4_{b}$}

  \put(6.6,4.6){$4_{c}$} 
  \put(6.6,0.2){$4_{c}$}

  \put(8.6,4.6){$4_{d}$} 
  \put(8.6,0.2){$4_{d}$}

  \put(4.6, -0.3){Figure 1. $g=1$}
  \end{picture}
  \end{center}
  \end{figure}
   \begin{figure}[h]
  \begin{center}
  \setlength\unitlength{1truecm}
  \begin{picture}(11.5,4.5)(0,0.5)


  \put(1.75,2.5){\vector(1,0){2}}
  \put(5.75,2.5){\vector(-1,0){2}}
  \put(5.75,2.5){\vector(1,0){2}}
  \put(9.75,2.5){\vector(-1,0){2}}
  
  \put(3.75,0.5){\vector(-1,0){1}}
  \put(3.75,0.5){\vector(1,0){1}}
  \put(7.75,0.5){\vector(-1,0){1}}
  \put(7.75,0.5){\vector(1,0){1}}
  
  \put(2.75,0.5){\vector(-1,0){1}}
  \put(4.75,0.5){\vector(1,0){1}}
  \put(6.75,0.5){\vector(-1,0){1}}
  \put(8.75,0.5){\vector(1,0){1}}
  
  \put(3.75,4.5){\vector(-1,0){1}}
  \put(3.75,4.5){\vector(1,0){1}}
  \put(7.75,4.5){\vector(-1,0){1}}
  \put(7.75,4.5){\vector(1,0){1}}
  
   \put(2.75,4.5){\vector(-1,0){1}}
  \put(4.75,4.5){\vector(1,0){1}}
  \put(6.75,4.5){\vector(-1,0){1}}
  \put(8.75,4.5){\vector(1,0){1}}

  \put(1.75,4.5){\vector(0,-1){2}}
  \put(3.75,2.5){\vector(0,1){2}}
  \put(5.75,4.5){\vector(0,-1){2}}
  \put(7.75,2.5){\vector(0,1){2}}
  \put(9.75,4.5){\vector(0,-1){2}}
  
  \put(1.75,0.5){\vector(0,1){2}}
  \put(3.75,2.5){\vector(0,-1){2}}
  \put(5.75,0.5){\vector(0,1){2}}
  \put(7.75,2.5){\vector(0,-1){2}}
  \put(9.75,0.5){\vector(0,1){2}}

  \put(2.05,3.5){$(+,+,+)$} 
  \put(4.05,3.5){$(-,-,+)$}
  \put(6.05,3.5){$(+,+,-)$}
  \put(8.05,3.5){$(-,-,-)$}  
  
  \put(2.05,1.5){$(-,+,-)$} 
  \put(4.05,1.5){$(+,-,-)$} 
  \put(6.05,1.5){$(-,+,+)$} 
  \put(8.05,1.5){$(+,-,+)$}

  \put(1.77,3.5){$1$} 
  \put(5.55,3.5){$1$} 
  \put(5.77,3.5){$1$} 
  \put(9.55,3.5){$1$} 
  \put(1.77,1.5){$1$} 
  \put(5.55,1.5){$1$} 
  \put(5.77,1.5){$1$} 
  \put(9.55,1.5){$1$}

  \put(2.6,2.55){$2$} 
  \put(4.6,2.55){$2$} 
  \put(6.6,2.55){$2$} 
  \put(8.6,2.55){$2$} 
  \put(2.6,2.2){$2$} 
  \put(4.6,2.2){$2$} 
  \put(6.6,2.2){$2$} 
  \put(8.6,2.2){$2$}

  \put(3.55,3.5){$3$} 
  \put(3.77,3.5){$3$} 
  \put(7.55,3.5){$3$} 
  \put(7.77,3.5){$3$} 
  \put(3.55,1.5){$3$} 
  \put(3.77,1.5){$3$} 
  \put(7.55,1.5){$3$} 
  \put(7.77,1.5){$3$}

  \put(3.2,4.2){$4$} 
  \put(5.2,4.2){$5$} 
  \put(7.2,4.2){$4$} 
  \put(9.2,4.2){$5$} 
  \put(3.2,0.55){$4$} 
  \put(5.2,0.55){$5$} 
  \put(7.2,0.55){$4$} 
  \put(9.2,0.55){$5$} 
  
   \put(2.3,4.2){$5$} 
  \put(4.3,4.2){$4$} 
  \put(6.3,4.2){$5$} 
  \put(8.3,4.2){$4$} 
  \put(2.3,0.55){$5$} 
  \put(4.3,0.55){$4$} 
  \put(6.3,0.55){$5$} 
  \put(8.3,0.55){$4$}

  \put(1.35,3.5){$1_{a}$} 
  \put(9.82,3.5){$1_{a}$}

  \put(1.35,1.5){$1_{b}$} 
  \put(9.82,1.5){$1_{b}$}

  \put(3.2,4.6){$4_{a}$} 
  \put(3.2,0.2){$4_{a}$}

  \put(5.2,4.6){$5_{b}$} 
  \put(5.22,0.2){$5_{d}$}

  \put(7.2,4.6){$4_{c}$} 
  \put(7.2,0.2){$4_{c}$}

  \put(9.2,4.6){$5_{d}$} 
  \put(9.2,0.2){$5_{b}$}

  \put(2.3,4.6){$5_{a}$} 
  \put(2.3,0.2){$5_{c}$}

  \put(4.3,4.6){$4_{b}$} 
  \put(4.3,0.2){$4_{b}$}

  \put(6.3,4.6){$5_{c}$} 
  \put(6.3,0.2){$5_{a}$}

  \put(8.3,4.6){$4_{d}$} 
  \put(8.3,0.2){$4_{d}$}

  \put(4.6, -0.3){Figure 2. $g=2$}
  \end{picture}
  \end{center}
  \end{figure}



\bibliographystyle{amsalpha}

\begin{thebibliography}{99}

\bibitem{Behrndt}
T. Behrndt.
\newblock Generalized {L}agrangian mean curvature flow in {K}\"ahler manifolds that are almost {E}instein.
\newblock In {\em Complex and {D}ifferential {G}eometry}, volume~8 of {\em Springer {P}roceedings in {M}athematics}, pages 65--79. Springer-{V}erlag, 2011.

\bibitem{Behrndt2}
T. Behrndt.
\newblock Mean curvature flow of {L}agrangian submanifolds with isolated conical singularities.
\newblock {\em \rm{arXive:1107.4803vl}}, 2011.

\bibitem{BoyerGalicki}
C. Boyer and K. Galicki.
\newblock 3-{S}asakian manifolds.
\newblock In {\em Surveys in differential geometry: essays on {E}instein manifolds}, Surv. Differ. Geom., VI, pages 123--184. Int. Press, Boston, MA, 1999.

\bibitem{ChoFutakiOno}
K. Cho, A. Futaki, and H. Ono.
\newblock Uniqueness and examples of compact toric {S}asaki-{E}instein metrics.
\newblock {\em Comm. Math. Phys.}, 277(2):439--458, 2008.

\bibitem{FutakiHattoriYamamoto}
A. Futaki, K. Hattori, and H. Yamamoto.
\newblock Self-similar solutions to the mean curvature flows on {R}iemannian cone manifolds and special {L}agrangians on toric {C}alabi--{Y}au cones.
\newblock {\em preprint}.

\bibitem{FutakiOnoWang}
A. Futaki, H. Ono, and G. Wang.
\newblock Transverse {K}\"ahler geometry of {S}asaki manifolds and toric {S}asaki-{E}instein manifolds.
\newblock {\em J. Differential Geom.}, 83(3):585--635, 2009.

\bibitem{Guillemin}
V. Guillemin.
\newblock Kaehler structures on toric varieties.
\newblock {\em J. Differential Geom.}, 40(2):285--309, 1994.

\bibitem{HarveyLawson}
R. Harvey and H.~B. Lawson, Jr.
\newblock Calibrated geometries.
\newblock {\em Acta Math.}, 148:47--157, 1982.

\bibitem{Huisken}
G. Huisken.
\newblock Asymptotic behavior for singularities of the mean curvature flow.
\newblock {\em J. Differential Geom.}, 31(1):285--299, 1990.

\bibitem{Joyce}
D. Joyce.
\newblock Special {L}agrangian {$m$}-folds in {$\Bbb C\sp m$} with symmetries.
\newblock {\em Duke Math. J.}, 115(1):1--51, 2002.

\bibitem{JoyceLeeTsui}
D. Joyce, Y.-I. Lee, and M.-P. Tsui.
\newblock Self-similar solutions and translating solitons for {L}agrangian mean curvature flow.
\newblock {\em J. Differential Geom.}, 84(1):127--161, 2010.

\bibitem{Lerman}
E. Lerman.
\newblock Contact toric manifolds.
\newblock {\em J. Symplectic Geom.}, 1(4):785--828, 2003.

\bibitem{MartelliSparksYau2}
D. Martelli, J. Sparks, and S.-T. Yau.
\newblock The geometric dual of {$a$}-maximisation for toric {S}asaki-{E}instein manifolds.
\newblock {\em Comm. Math. Phys.}, 268(1):39--65, 2006.

\bibitem{MartelliSparksYau}
D. Martelli, J. Sparks, and S.-T. Yau.
\newblock Sasaki-{E}instein manifolds and volume minimisation.
\newblock {\em Comm. Math. Phys.}, 280(3):611--673, 2008.

\bibitem{Oda}
T. Oda. 
\newblock Convex bodies and algebraic geometry. {A}n introduction to the theory of toric varieties. 
\newblock Ergebnisse der Mathematik und ihrer Grenzgebiete 15. Berlin, Springer-Verlag, 1988.

\bibitem{SmoczykWang}
K. Smoczyk and M.-T. Wang.
\newblock Generalized {L}agrangian mean curvature flows in symplectic manifolds.
\newblock {\em Asian J. Math.}, 15(1):129--140, 2011.

\bibitem{StromingerYauZaslow}
A. Strominger, S.-T. Yau, and E. Zaslow.
\newblock Mirror symmetry is {$T$}-duality.
\newblock {\em Nuclear Phys. B}, 479(1-2):243--259, 1996.

\bibitem{ThomasYau}
R.~P. Thomas and S.-T. Yau.
\newblock Special {L}agrangians, stable bundles and mean curvature flow.
\newblock {\em Comm. Anal. Geom.}, 10(5):1075--1113, 2002.

\end{thebibliography}

\end{document}